\documentclass[a4paper,10pt]{article}
\usepackage{amssymb,amsmath}
\input xy
\xyoption{all}
\usepackage[all]{xy}
\usepackage{amsthm,amsmath,amsfonts,amssymb,mathrsfs}

\theoremstyle{plain}
\newtheorem{theorem}{Theorem}[section]
\newtheorem{lemma}[theorem]{Lemma}
\newtheorem{proposition}[theorem]{Proposition}
\newtheorem{corollary}[theorem]{Corollary}

\theoremstyle{definition}
\newtheorem{definition}[theorem]{Definition}

\theoremstyle{remark}
\newtheorem{remark}[theorem]{Remark}

\numberwithin{equation}{section}
\numberwithin{figure}{section}

\author{K\^az\i m \.Ilhan \.IKEDA and Erol SERBEST}

\title{Generalized Fesenko reciprocity map}

\date{}
\begin{document}

\maketitle

\begin{abstract}
In this paper, which is the natural continuation and generalization of
\cite{fesenko2000, fesenko2001, fesenko2005} and \cite{ikeda-serbest},
we extend the theory of Fesenko to infinite $APF$-Galois extensions $L$ 
over a local field $K$, with finite residue-class field $\kappa_K$ of $q=p^f$ 
elements, satisfying $\pmb{\mu}_p(K^{sep})\subset K$ and
$K\subset L\subset K_{\varphi^d}$ where the residue-class degree 
$[\kappa_L:\kappa_K]=d$. More precisely, for such extensions $L/K$,
fixing a Lubin-Tate splitting $\varphi$ over $K$, we construct a $1$-cocycle,
\begin{equation*}
\pmb{\Phi}_{L/K}^{(\varphi)}:\text{Gal}(L/K)\rightarrow 
K^\times/N_{L_0/K}L_0^\times\times U_{\widetilde{\mathbb X}(L/K)}^\diamond
/Y_{L/L_0},
\end{equation*} 
where $L_0=L\cap K^{nr}$, and study its functorial and ramification-theoretic
properties. The case $d=1$ recovers the theory of Fesenko. 



\end{abstract}

\maketitle
Let $K$ be a local field (that is, a complete discrete valuation field)
with finite residue-class field $\kappa_K$ of $q=p^f$ elements. 
Assume that $\pmb{\mu}_p(K^{sep})\subset K$. Fix a Lubin-Tate splitting 
$\varphi$ over $K$ (cf. \cite{koch-deshalit}).
In \cite{fesenko2000, fesenko2001, fesenko2005}, Fesenko introduced a 
very general non-abelian local reciprocity map  
\begin{equation*}
\Phi_{L/K}^{(\varphi)}:\text{Gal}(L/K)\rightarrow 
U_{\widetilde{\mathbb X}(L/K)}^\diamond/Y_{L/K}
\end{equation*}
defined for any totally-ramified infinite $APF$-Galois extension $L/K$ 
satisfying $K\subset L\subset K_\varphi$, which generalizes the previous
non-abelian local class field theories of Koch-de Shalit \cite{koch-deshalit}
and Gurevich \cite{gurevich}. In \cite{ikeda-serbest}, we have studied the
basic functorial and ramification-theoretic properties of the reciprocity 
map of Fesenko.
 
In this paper, which is the natural continuation and generalization of
\cite{fesenko2000, fesenko2001, fesenko2005} and \cite{ikeda-serbest},
we extend the theory of Fesenko to infinite $APF$-Galois extensions $L/K$ 
satisfying $K\subset L\subset K_{\varphi^d}$ where the residue-class degree 
$[\kappa_L:\kappa_K]=d$. More precisely, for such extensions $L/K$,
we construct a $1$-cocycle,
\begin{equation*}
\pmb{\Phi}_{L/K}^{(\varphi)}:\text{Gal}(L/K)\rightarrow 
K^\times/N_{L_0/K}L_0^\times\times U_{\widetilde{\mathbb X}(L/K)}^\diamond
/Y_{L/L_0},
\end{equation*} 
where $L_0=L\cap K^{nr}$, and study its functorial and ramification-theoretic
properties. Note that, the case $d=1$ recovers the theory of Fesenko.

The organization of this paper is as follows. In the first section, we briefly
review the necessary background material from Fontaine-Wintenberger theory of 
fields of norms. In the second section, we introduce the generalized Fesenko
reciprocity map $\pmb{\Phi}_{L/K}^{(\varphi)}$ of an extension $L/K$, which
is an infinite $APF$-Galois extension satisfying $K\subset L\subset 
K_{\varphi^d}$ where the residue-class degree $[\kappa_L:\kappa_K]=d$, and
study its functorial and ramification-theoretic properties.

The material and results of this paper play a fundamental role in our 
construction of non-abelian local class field theory \cite{ikeda-serbest-2}, 
which generalizes Laubie theory \cite{laubie} as well. 

\subsection*{Acknowledgements}
The first named author (K. I. I) would like to thank to the Institut de
Math{\'e}matiques, ``Th{\'e}orie des Groupes, Repr{\'e}sentations et
Applications'', Universit{\'e} Paris 7, Jussieu, Paris, France, and to
the School of Pure Mathematics of the Tata Institute of Fundamental
Research, Mumbai, India, for their hospitality and support, where
some parts of this work has been completed.
The second named author (E. S.) would like to thank to T\"UBITAK for
a fellowship, and to the School of Mathematics of the University of Nottingham
for the hospitality and support, where some parts of this work have been
introduced.
Both of the authors would like to thank I. B. Fesenko for his
interest and encouragement at all stages of this work.

\subsection*{Notation}
All through this work, $K$ will denote a local field (a complete
discrete valuation field) with finite residue
field $O_K/\mathfrak p_K=:\kappa_K$ of $q_K=q=p^f$ elements with $p$ a
prime number, where $O_K$
denotes the ring of integers in $K$ with the unique maximal 
ideal $\mathfrak p_K$.
Let $\pmb{\nu}_K$ denote the corresponding normalized valuation on $K$
(normalized by $\pmb{\nu}_K(K^\times)=\mathbb Z$), and
$\widetilde{\pmb{\nu}}$ the unique extension of $\pmb{\nu}_K$ to a fixed
separable closure $K^{sep}$ of $K$. For any sub-extension $L/K$ of
$K^{sep}/K$, the normalized form of the
valuation $\widetilde{\pmb{\nu}}\mid_L$ on $L$ will be denoted by
$\pmb{\nu}_L$.
As usual, we let $K^{nr}$ to denote the maximal unramified extension in 
$K^{sep}$, and $\widetilde{K}$ denotes the completion of $K^{nr}$ with
respect to $\pmb{\nu}_{K^{nr}}$. Fix a \textit{Lubin-Tate splitting 
$\varphi_K=\varphi$ 
over $K$}. The fixed field of the Lubin-Tate splitting $\varphi$ is denoted
by $K_\varphi$. Finally, let 
$\left(\pi_E\right)_{K\subset E\subset K_\varphi}$ 
be the canonical sequence of norm-compatible prime elements in finite 
sub-extensions $E/K$ in $K_\varphi/K$. This determines a unique 
\textit{Lubin-Tate labelling over $K$} (cf. paragraph 0.2 of 
\cite{koch-deshalit}).

\section{Preliminaries on Fontaine-Wintenberger field of norms}
For a brief review of $APF$-extensions and Fontaine-Wintenberger field
of norms, we refer the reader to \cite{ikeda-serbest}, and for detailed
proofs to \cite{fontaine-wintenberger1979-a, fontaine-wintenberger1979-b,
wintenberger1983}.

Let $L/K$ be an infinite Galois arithmetically profinite (in short $APF$)
extension such that the residue-class degree $[\kappa_L : \kappa_K]=d$ and
$K\subset L\subset K_{\varphi^d}$; that is, in the terminology of
Koch-de Shalit in \cite{koch-deshalit} and Laubie in \cite{laubie}, 
$L$ is \textit{compatible with} 
$(T,\varphi)$, where $T$ denotes the intersection field
$L\cap K_\varphi$. Note that, $T/K$ is in general \textit{not} a normal
extension! 
Denote by $L_0^{(K)}=L\cap K^{nr}=K_d^{nr}$. If there is no fear of
confusion, denote $L_0^{(K)}$ simply by $L_0$. So, there is the following
diagram :
\begin{equation*}
\xymatrix{
{} & {} & {} & {} & {K^{sep}}\ar@{-}[d] \\
{} & {} & {} & {} & {K_{\varphi^d}}\ar@{-}[d] \\
{K^{nr}}\ar@{-}[uurrrr]\ar@{-}[d] & {} & L\ar@{-}[urr]\ar@{-}[d] & {} & 
{K_\varphi} \\
{K^{nr}_d}\ar@{-}[urr]\ar@{-}[d] & {} & T\ar@{-}[urr] & {} & {} \\
K\ar@{-}[urr] & {} & {} & {} & {} 
}
\end{equation*}
\begin{remark}
\label{key-remark}
Note that, $\varphi'=\varphi^d$ is a Lubin-Tate splitting over $L_0=K^{nr}_d$. 
Thus, by Proposition 1.2.3 of \cite{wintenberger1983} or by Lemma 3.3 of 
\cite{ikeda-serbest}, $L/L_0$ is an infinite totally-ramified $APF$-Galois 
extension satisfying
$L_0\subseteq L\subseteq (L_0)_{\varphi'}$. Thus, Fesenko theory, developed
in \cite{fesenko2000, fesenko2001, fesenko2005} and \cite{ikeda-serbest},
works for the extension $L/L_0$.
\end{remark}
Since $L/T$ is an unramified extension, we have the following lemma. 
\begin{lemma}
\label{unramified-L/L_0-over-T/K}
The field of norms $\mathbb X(L/L_0)$
is an unramified extension of the field of norms $\mathbb X(T/K)$. 
\end{lemma}
\begin{proof}
In fact, there exists a natural isomorphism 
$\mathbb X(L/L_0)\xrightarrow{\sim}\mathbb X(L/K)$
which identifies $\mathbb X(L/L_0)$ and $\mathbb X(L/K)$ (cf. section (5.6) of
Chapter III of \cite{fesenko-vostokov}). 
Now, $\mathbb X(L/K)$ is a Galois extension of $\mathbb X(T/K)$ with 
corresponding Galois group isomorphic to $\text{Gal}(L/T)$ (cf. 
\cite{ikeda-serbest} and \cite{wintenberger1983}).
As $\kappa_{\mathbb X(L/K)}\simeq\kappa_L$ and 
$\kappa_{\mathbb X(T/K)}\simeq\kappa_T$, it follows that
\begin{equation*}
[\kappa_{\mathbb X(L/K)}:\kappa_{\mathbb X(T/K)}]=
[\mathbb X(L/K):\mathbb X(T/K)],
\end{equation*} 
as $L/T$ is an unramified extension, which proves that $\mathbb X(L/K)$ is an
unramified extension of $\mathbb X(T/K)$.
\end{proof}
Now, as the Lubin-Tate splitting $\varphi$ over $K$ is fixed, the
unique element $\Pi_{\varphi;T/K}=(\pi_E)_{K\subset E\subset T}
\in\mathbb X(T/K)$ is a canonical prime element of the local field 
$\mathbb X(T/K)$. Thus, in view of Lemma 
\ref{unramified-L/L_0-over-T/K}, $\Pi_{\varphi;T/K}$ is a prime element of 
$\mathbb X(L/L_0)$ as well. Moreover,
\begin{lemma}
\label{lemma-prime-equality}
\begin{equation}
\label{prime-equality}
\Pi_{\varphi;T/K} = \Pi_{\varphi';L/L_0} .
\end{equation}
\end{lemma}
\begin{proof}
In fact, for the Lubin-Tate splitting $\varphi'=\varphi^d$ over $L_0=K_d^{nr}$,
there exists a unique element 
$(\pi_{L_0E})_{L_0\subset L_0E\subset L_0T=L}\in\mathbb X(L/L_0)$. As $L_0E/E$
is an unramified extension, it follows that $\pi_{L_0E}=\pi_E$ for each
$K\subset E\subset T$. Thus, eq. (\ref{prime-equality}) follows.
\end{proof}
The completion $\widetilde{\mathbb X}(L/K)$ of the maximal unramified extension
$\mathbb X(L/K)^{nr}$ of the field of norms $\mathbb X(L/K)$ is identified
with the field of norms $\mathbb X(\widetilde{L}/\widetilde{K})
=\mathbb X(\widetilde{L}/\widetilde{L_0})$.

\section{Generalized Fesenko reciprocity map}
The main references for this section are \cite{fesenko2000,
fesenko2001, fesenko2005} and \cite{ikeda-serbest}.

Fix a Lubin-Tate splitting $\varphi_K=\varphi$ over $K$.
The aim of this section is to generalize the reciprocity map 
$\Phi_{M/K}^{(\varphi)}$ of Fesenko, cf. \cite{fesenko2000,
fesenko2001, fesenko2005} and \cite{ikeda-serbest}, defined for
infinite totally-ramified $APF$-Galois extensions $M/K$ satisfying 
$K\subset M\subset K_\varphi$ to infinite $APF$-Galois extensions $L/K$
with residue-class degree $[\kappa_L:\kappa_K]=d$ and satisfying
$K\subset L\subset K_{\varphi^d}$.
We shall keep the notation introduced in \cite{ikeda-serbest} and introduced 
in the previous section in what follows. 

Recall that, for the extension $M/K$ as above, the diamond subgroup 
$U_{\widetilde{\mathbb X}(M/K)}^\diamond$ of the group 
$U_{\widetilde{\mathbb X}(M/K)}$ of units of the ring of integers of
$\widetilde{\mathbb X}(M/K)$ is defined by
\begin{equation*}
U_{\widetilde{\mathbb X}(M/K)}^\diamond = \text{Pr}_{\widetilde{K}}^{-1}
\left(U_K\right),
\end{equation*}
where $\text{Pr}_{\widetilde{K}}:U_{\widetilde{\mathbb X}(M/K)}
\rightarrow U_{\widetilde{K}}$ denotes the projection map on the 
$\widetilde{K}$-coordinate of $U_{\widetilde{\mathbb X}(M/K)}$.
More generally, for a given infinite $APF$-Galois extension $L/K$ with
residue-class degree $[\kappa_L : \kappa_K]=d$ and satisfying 
$K\subset L\subset K_{\varphi^d}$, define the diamond subgroup 
$U_{\widetilde{\mathbb X}(L/K)}^\diamond$ of the group 
$U_{\widetilde{\mathbb X}(L/K)}$ of units of the ring of integers of 
$\widetilde{\mathbb X}(L/K)= \widetilde{\mathbb X}(L/L_0)$ naturally 
as follows. 
\begin{definition}
\label{diamond-group-L/K}
$U_{\widetilde{\mathbb X}(L/K)}^\diamond$ is the subgroup
of the group $U_{\widetilde{\mathbb X}(L/K)}$ of units of the ring of integers
of the local field $\widetilde{\mathbb X}(L/K)$ whose 
$\widetilde{K}=\widetilde{L}_0$-coordinate belongs to $U_{L_0}$. That is,
\begin{equation*}
U_{\widetilde{\mathbb X}(L/K)}^\diamond =
U_{\widetilde{\mathbb X}(L/L_0)}^\diamond .
\end{equation*}  
\end{definition}
Now, as a first step, we shall generalize the arrow $\phi_{M/K}^{(\varphi)}$
defined for the extensions $M/K$, where $M/K$ is a totally-ramified
$APF$-Galois extension satisfying $K\subset M\subset K_\varphi$, 
of Fesenko theory,
which has been described in \cite{fesenko2000, fesenko2001, fesenko2005}
and in detail in Section 5 of \cite{ikeda-serbest}, to infinite $APF$-Galois 
extensions $L$ of $K$ with residue-class degree $d$ and satisfying 
$K^{nr}_{d}\subset L\subset K_{\varphi^d}$ and
construct the \textit{generalized arrow ${\pmb\phi}_{L/K}^{(\varphi)}$ 
for such extensions $L/K$} as follows.
There exists an isomorphism
\begin{equation}
\label{directproduct}
\text{Gal}(L/K)\xrightarrow{\sim}\text{Gal}(L_0/K)\times\text{Gal}(L/L_0)
\end{equation}
defined by
\begin{equation}
\label{directproduct-definition}
\sigma\mapsto (\sigma\mid_{L_0},\varphi^{-m}\sigma),
\end{equation}
for every $\sigma\in\text{Gal}(L/K)$, where $\sigma\mid_{L_0}=\varphi^m$ for
some $0\leq m\in\mathbb Z$. 
\begin{remark}
\label{directproduct-diagram}
\begin{itemize}
\item[(i)]
Let $M/K$ be a Galois sub-extension of $L/K$. Let $M_0=M\cap K^{nr}$.
Then, the following square
\begin{equation*}
\SelectTips{cm}{}\xymatrix{
{\text{Gal}(L/K)}\ar[d]_{res_M}\ar[r]^-{\sim} & 
{\text{Gal}(L_0/K)\times\text{Gal}(L/L_0)}\ar@{.>}[d]^{(res_{M_0},res_M)} \\
{\text{Gal}(M/K)}\ar[r]^-{\sim} & {\text{Gal}(M_0/K)\times\text{Gal}(M/M_0)}
}
\end{equation*}
is commutative. Now, for $\sigma\in\text{Gal}(L/K)$, there exists
$0\leq m,m'\in\mathbb Z$ such that $\sigma\mid_{L_0}=\varphi^m$ and
$(\sigma\mid_M)\mid_{M_0}=\varphi^{m'}$. Thus, $\varphi^m\mid_{M_0}=
\varphi^{m'}\mid_{M_0}$ and the identity
$(\varphi^{-m}\sigma)\mid_M=\varphi^{-m'}(\sigma\mid_M)$
is satisfied. 
\item[(ii)]
Let $F/K$ be a finite Galois sub-extension of $L/K$. Let 
$L_0^{(K)}=L\cap K^{nr}$ and $L_0^{(F)}=L\cap F^{nr}$. Then, the following
square
\begin{equation*}
\SelectTips{cm}{}\xymatrix{
{\text{Gal}(L/F)}\ar[d]_{\text{inc.}}\ar[r]^-{\sim} & 
{\text{Gal}(L_0^{(F)}/F)\times\text{Gal}(L/L_0^{(F)})}
\ar@{.>}[d]^{(res_{L_0^{(K)}},\text{inc.})} \\
{\text{Gal}(L/K)}\ar[r]^-{\sim} & {\text{Gal}(L_0^{(K)}/K)\times
\text{Gal}(L/L_0^{(K)})}
}
\end{equation*}
is commutative. Now, for any $\sigma\in\text{Gal}(L/F)$, there exists
$0\leq m,m'\in\mathbb Z$ such that
$\sigma\mid_{L_0^{(F)}}=\varphi_F^m$ and 
$\sigma\mid_{L_0^{(K)}}=\varphi_K^{m'}$. Thus, 
$\varphi_F^m\mid_{L_0^{(K)}}=\varphi_K^{m'}$ and the identity 
$\varphi_F^{-m}\sigma=\varphi_K^{-m'}\sigma$ is satisfied.
\end{itemize}
\end{remark}
Now, by Proposition 1.2.3 of 
\cite{wintenberger1983} or by Lemma 3.3 of \cite{ikeda-serbest}, $L/L_0$ 
is a totally-ramified $APF$-Galois extension with 
$L_0\subset L\subset (L_0)_{\varphi'}$, where $\varphi'=\varphi^d$ is a 
Lubin-Tate splitting over $L_0$ by Remark \ref{key-remark}. Thus, define 
the map
\begin{equation}
\label{generalized-arrow}
{\pmb\phi}_{L/K}^{(\varphi)} :
\text{Gal}(L/K)\rightarrow K^\times/N_{L_0/K}L_0^\times
\times U_{\widetilde{\mathbb X}(L/K)}^\diamond/U_{\mathbb X(L/K)}
\end{equation}
by
\begin{equation}
\label{generalized-arrow-definition}
{\pmb\phi}_{L/K}^{(\varphi)}(\sigma):=
\left(\pi_K^m . N_{L_0/K}L_0^\times ,
(u_{\widetilde{E}}).U_{\mathbb X(L/K)}\right),
\end{equation}
where $\sigma\in\text{Gal}(L/K)$ such that
$\sigma\mid_{L_0}=\varphi^m$, for some $0\leq m\in\mathbb Z$, and 
$U=(u_{\widetilde{E}})\in U_{\widetilde{\mathbb X}(L/L_0)}$ satisfies the
equality
\begin{equation}
\label{main-equation}
U^{1-\varphi^d}=\Pi_{\varphi';L/L_0}^{(\varphi^{-m}\sigma)-1},
\end{equation}
where $\Pi_{\varphi';L/L_0}$ is the canonical prime element
of the local field $\mathbb X(L/L_0)$, which is the canonical prime element 
$\Pi_{\varphi;T/K}$ of the local field $\mathbb X(T/K)$ by Lemma 
\ref{unramified-L/L_0-over-T/K} and by Lemma \ref{lemma-prime-equality}.
Thus, (\ref{main-equation}) can be reformulated by
\begin{equation}
\label{main-equation-alternative}
U^{1-\varphi^d}=\Pi_{\varphi';L/L_0}^{\sigma-1},
\end{equation}
as $\Pi_{\varphi;T/K}^\varphi=\Pi_{\varphi;T/K}$. Moreover, the solution
$U=\left(u_{\widetilde{E}}\right)\in U_{\widetilde{\mathbb X}(L/L_0)}$, which
is unique modulo $U_{\mathbb X(L/K)}$, satisfies $\text{Pr}_{L_0}(U)
=u_{\widetilde{L}_0}\in U_{L_0}$. In fact, by Lemma \ref{lemma-prime-equality},
$\text{Pr}_{L_0}(\Pi_{\varphi';L/L_0})=\pi_K$, and therefore
$\text{Pr}_{L_0}(\Pi_{\varphi';L/L_0}^{\sigma-1})=\pi_{K}^{\sigma-1}=1_K$.
Hence, $\text{Pr}_{L_0}(U^{1-\varphi^d})
=\text{Pr}_{L_0}(\Pi_{\varphi';L/L_0}^{\sigma-1})=1_K$ yields 
$u_{\widetilde{L}_0}^{1-\varphi^d}=1_K$, that is $u_{\widetilde{L}_0}\in
U_{L_0}$ as $\widetilde{L}_0\cap (L_0)_{\varphi'}=L_0$. Now, it follows that,
$\text{Pr}_{L_0}(U)=u_{\widetilde{L}_0}\in U_{L_0}$. Thus, 
$U=(u_{\widetilde{E}})$ belongs to $U_{\widetilde{\mathbb X}(L/K)}^\diamond$,
by Definition \ref{diamond-group-L/K}.
\begin{remark}
\label{generalized-arrow-definition-new}
We can reformulate the definition of the generalized arrow
\begin{equation*}
{\pmb\phi}_{L/K}^{(\varphi)}:\text{Gal}(L/K)
\rightarrow K^\times/N_{L_0/K}L_0^\times
\times U_{\widetilde{\mathbb X}(L/K)}^\diamond/U_{\mathbb X(L/K)}
\end{equation*}
for the extension $L/K$ as
\begin{equation*}
{\pmb\phi}_{L/K}^{(\varphi)}(\sigma)=\left(\pi_K^m.N_{L_0/K}L_0^\times,
\phi_{L/L_0}^{(\varphi')}(\varphi^{-m}\sigma)\right)
\end{equation*}
for every $\sigma\in\text{Gal}(L/K)$, where $\sigma\mid_{L_0}=\varphi^m$
for some $0\leq m\in\mathbb Z$.  
\end{remark}
There is a natural continuous action of $\text{Gal}(L/K)$ on the
topological group  
$K^\times/N_{L_0/K}L_0^\times\times U_{\widetilde{\mathbb X}(L/K)}^\diamond
/U_{\mathbb X(L/K)}$ 
defined by abelian local class field theory on
the first component and by eq.s (5.5) and (5.7) of \cite{ikeda-serbest}
on the second component as
\begin{equation}
\label{galoismodule1}
(\overline{a},\overline{U})^\sigma=
\left(\overline{a}^{\varphi^m},\overline{U}^{\varphi^{-m}\sigma}\right)=
\left(\overline{a},\overline{U}^{\varphi^{-m}\sigma}\right),
\end{equation}
for every $\sigma\in\text{Gal}(L/K)$, where $\sigma\mid_{L_0}=\varphi^m$
for some $0\leq m\in\mathbb Z$, and
for every $a\in K^\times$ with $\overline{a}=a.N_{L_0/K}L_0^\times$ 
and 
$U\in U_{\widetilde{\mathbb X}(L/K)}^\diamond$ with
$\overline{U}=U.U_{\mathbb X(L/K)}$.
\textit{We shall always view 
$K^\times/N_{L_0/K}L_0^\times
\times U_{\widetilde{\mathbb X}(L/K)}^\diamond/U_{\mathbb X(L/K)}$
as a topological $\text{Gal}(L/K)$-module in this text}.
\begin{theorem}
\label{cocycle1}
Let $L/K$ be any infinite $APF$-Galois sub-extension of $K_{\varphi^d}/K$
with residue-class degree $d$. Then the generalized arrow
\begin{equation*}
{\pmb\phi}_{L/K}^{(\varphi)}:\text{Gal}(L/K)\rightarrow K^\times/
N_{L_0/K}L_0^\times\times U_{\widetilde{\mathbb X}(L/K)}^\diamond/
U_{\mathbb X(L/K)} 
\end{equation*} 
defined for the extension $L/K$ is an injection, and for every
$\sigma,\tau\in\text{Gal}(L/K)$, 
\begin{equation}
{\pmb\phi}_{L/K}^{(\varphi)}(\sigma\tau)={\pmb\phi}_{L/K}^{(\varphi)}(\sigma)
{\pmb\phi}_{L/K}^{(\varphi)}(\tau)^\sigma
\end{equation} 
co-cycle condition is satisfied.
\end{theorem}
\begin{proof}
The injectivity of the arrow given by eq. (\ref{generalized-arrow}) 
and defined by eq. (\ref{generalized-arrow-definition}) is clear from the 
canonical topological isomorphism defined by
(\ref{directproduct}) and (\ref{directproduct-definition})
combined with abelian local class field theory and Theorem
5.6 of Fesenko in \cite{ikeda-serbest}. 
To be precise, let ${\pmb\phi}_{L/K}^{(\varphi)}(\sigma)
=(\pi_K^m,(u_{\widetilde{E}}))$ with $d\mid m$ and $(u_{\widetilde{E}})
\in U_{\mathbb X(L/L_0)}=U_{\mathbb X(L/K)}$. As $d\mid m$, $\sigma$ acts
trivially on $L_0$. Since 
$(u_{\widetilde{E}})^{\varphi^d-1}=(1_{\widetilde{E}})=
\Pi_{\varphi';L/L_0}^{\sigma-1}$, $\sigma$ acts trivially on the prime 
elements of finite sub-extensions between $L_0$ and $L$. Thus, $\sigma$
is the identity element of $\text{Gal}(L/L_0)$. 
Now, for
$\sigma,\tau\in\text{Gal}(L/K)$, with $\sigma\mid_{L_0}=\varphi^m$ and
$\tau\mid_{L_0}=\varphi^n$ for some $0\leq m,n\in\mathbb Z$, following the
alternative definition of the generalized arrow ${\pmb\phi}_{L/K}^{(\varphi)}$
introduced in Remark \ref{generalized-arrow-definition-new},
\begin{equation*}
\begin{aligned}
{\pmb\phi}_{L/K}^{(\varphi)}(\sigma\tau) & =
(\pi_K^{m+n}.N_{L_0/K}L_0^\times,\phi_{L/L_0}^{(\varphi')}(\varphi^{-(m+n)}
\sigma\tau))\\
& =\left((\pi_K^m.N_{L_0/K}L_0^\times)(\pi_K^n.N_{L_0/K}L_0^\times),
\phi_{L/L_0}^{(\varphi')}(\varphi^{-m}\sigma)
\phi_{L/L_0}^{(\varphi')}(\varphi^{-n}\tau)^{\varphi^{-m}\sigma}\right)\\
& =\left(\pi_K^m.N_{L_0/K}L_0^\times ,\phi_{L/L_0}^{(\varphi')}
(\varphi^{-m}\sigma)\right)
\left(\pi_K^n.N_{L_0/K}L_0^\times,\phi_{L/L_0}^{(\varphi')}(\varphi^{-n}\tau)
^{\varphi^{-m}\sigma}\right)\\
&={\pmb\phi}_{L/K}^{(\varphi)}(\sigma){\pmb\phi}_{L/K}^{(\varphi)}(\tau)^\sigma
\end{aligned}
\end{equation*} 
by Theorem 5.6 of Fesenko in \cite{ikeda-serbest} and by the definition of the
action of $\sigma\in\text{Gal}(L/K)$ on 
${\pmb\phi}_{L/K}^{(\varphi)}(\tau)\in K^\times/
N_{L_0/K}L_0^\times
\times U_{\widetilde{\mathbb X}(L/K)}^\diamond/
U_{\mathbb X(L/K)}$
defined by eq. (\ref{galoismodule1}).
\end{proof}
Now, we immediately have the following result.
\begin{corollary}
\label{grouplaw1}
Define a law of composition $\ast$ on 
$\text{im}({\pmb\phi}_{L/K}^{(\varphi)})$ by
\begin{equation}
\label{star1}
(\overline{a},\overline{U})\ast (\overline{b},\overline{V})
=(\overline{a}.\overline{b},\overline{U}.
\overline{V}^{(\phi_{L/L_0}^{(\varphi')})^{-1}(\overline{U})}),
\end{equation}
for $(\overline{a},\overline{U})$, $(\overline{b},\overline{V})\in
\text{im}({\pmb\phi}_{L/K}^{(\varphi)})$, where
$\overline{a}=a.N_{L_0/K}L_0^\times, \overline{b}=b.N_{L_0/K}L_0^\times
\in K^\times/N_{L_0/K}L_0^\times$ with $a,b\in K^\times$ and
for $\overline{U}=U.U_{\mathbb X(L/K)}, \overline{V}=V.U_{\mathbb X(L/K)}\in 
U_{\widetilde{\mathbb X}(L/K)}^\diamond/U_{\mathbb X(L/K)}$ with $U,V\in
U_{\widetilde{\mathbb X}(L/K)}^\diamond$. Then
$\text{im}({\pmb\phi}_{L/K}^{(\varphi)})$ is a topological 
group under $\ast$, and the map ${\pmb\phi}_{L/K}^{(\varphi)}$ induces 
an isomorphism of topological groups
\begin{equation}
\label{isomorphism1}
{\pmb\phi}_{L/K}^{(\varphi)}:\text{Gal}(L/K)\xrightarrow{\sim}
\text{im}({\pmb\phi}_{L/K}^{(\varphi)}), 
\end{equation}
where the topological group structure on 
$\text{im}({\pmb\phi}_{L/K}^{(\varphi)})$
is defined with respect to the binary operation $\ast$ defined by eq. 
(\ref{star1}).
\end{corollary}
Recall that, for an infinite $APF$-Galois extension $L/K$ and for
every $-1\leq u\in\mathbb R$, the $u^{th}$ higher ramification 
subgroup $\text{Gal}(L/K)_{u}$ of $\text{Gal}(L/K)$ in lower numbering
is defined by
\begin{equation*}
\text{Gal}(L/K)_u=\text{Gal}(L/K)^{\varphi_{L/K}(u)},
\end{equation*}
where the number $-1\leq\varphi_{L/K}(u)\in\mathbb R$ is defined by
eq. (3.1) in \cite{ikeda-serbest} and as usual the $\varphi_{L/K}(u)^{th}$ 
higher ramification subgroup $\text{Gal}(L/K)^{\varphi_{L/K}(u)}$ 
of $\text{Gal}(L/K)$ in upper numbering by the projective limit
$\text{Gal}(L/K)^{\varphi_{L/K}(u)}=\varprojlim\limits_{K\subseteq
  F\subset L}\text{Gal}(F/K)^{\varphi_{L/K}(u)}$
defined by eq.s (2.1) and (2.2) in \cite{ikeda-serbest}. 
Now, let $E/K$ be a Galois sub-extension of $L/K$. Then, for any chain
of field extensions 
$\overbrace{\underbrace{K\subseteq F}_{\text{finite Gal.}}\subseteq F'}
^{\text{finite Gal.}}\subset L$, the square
\begin{equation}
\label{t-square}
\SelectTips{cm}{}\xymatrix{
{\text{Gal}(F'/K)^{\varphi_{L/K}(u)}}\ar[rr]^-{t_{F'\cap E}^{F'}
(\varphi_{L/K}(u))}
\ar[d]_{t_F^{F'}(\varphi_{L/K}(u))} & & {\text{Gal}(F'\cap
  E/K)^{\varphi_{L/K}(u)}}\ar[d]^{t_{F\cap E}^{F'\cap E}
(\varphi_{L/K}(u))} \\
{\text{Gal}(F/K)^{\varphi_{L/K}(u)}}\ar[rr]^-{t_{F\cap E}^F
(\varphi_{L/K}(u))} & &
{\text{Gal}(F\cap E/K)^{\varphi_{L/K}(u)}}
}
\end{equation}
is commutative. Thus, passing to the projective limits, there exists 
a continuous group homomorphism
\begin{equation}
\label{t-morphism}
t_E^L(\varphi_{L/K}(u))=
\varprojlim\limits_{K\subseteq F\subset L}
t_{F\cap E}^{F}(\varphi_{L/K}(u)) :
\text{Gal}(L/K)^{\varphi_{L/K}(u)}\rightarrow
\text{Gal}(E/K)^{\varphi_{L/K}(u)},
\end{equation}
which is essentially the restriction morphism from $L$ to $E$.
This morphism is a surjection, as the objects in the 
respective projective systems are compact and Hausdorff. Furthermore,
the following square
\begin{equation}
\label{r-t-square}
\SelectTips{cm}{}\xymatrix{
{\text{Gal}(L/K)}\ar[rr]^-{r_E^L} & & {\text{Gal}(E/K)} \\
{\text{Gal}(L/K)^{\varphi_{L/K}(u)}}\ar[rr]^-{t_E^L(\varphi_{L/K}(u))}
\ar[u]^{\text{inc.}} & &
{\text{Gal}(E/K)^{\varphi_{L/K}(u)}}\ar[u]_{\text{inc.}} \\
{\text{Gal}(L/K)^{\varphi_{L/K}(u')}}\ar[rr]^-{t_E^L(\varphi_{L/K}(u'))}
\ar[u]^{\text{inc.}} & &
{\text{Gal}(E/K)^{\varphi_{L/K}(u')}}\ar[u]_{\text{inc.}}
}
\end{equation}
is commutative for every pair $u,u'\in\mathbb R_{\geq -1}$ satisfying
$u\leq u'$. Here, the arrow $r_E^L : \text{Gal}(L/K)\rightarrow
\text{Gal}(E/K)$ denotes the restriction map.
Therefore, immediately, we have the following result.
\begin{lemma}
\label{directproduct-lemma}
For $0\leq u\in\mathbb R$,
the topological isomorphism defined by eq.s (\ref{directproduct}) and
(\ref{directproduct-definition}) induces a topological isomorphism
\begin{equation}
\label{directproduct2}
\text{Gal}(L/K)_u\simeq\underbrace{\text{Gal}(L_0/K)^{\varphi_{L/K}(u)}}
_{\left<\text{id}_{L_0}\right>}\times\text{Gal}(L/L_0)^{\varphi_{L/K}(u)}
\end{equation}
defined by
\begin{equation}
\label{directproduct-definition2}
\sigma\mapsto \left(t_{L_0}^L(\varphi_{L/K}(u))(\sigma),\varphi^{-m}\sigma
\right)=(\text{id}_{L_0},\sigma),
\end{equation}
for every $\sigma\in\text{Gal}(L/K)_u$ with $\sigma\mid_{L_0}=
t_{L_0}^L(\varphi_{L/K}(u))(\sigma)=\varphi^m\mid_{L_0}$ for some 
$0\leq m\in\mathbb Z$ satisfying $d\mid m$. 
\end{lemma}
\begin{proof}
Note that, $\text{Gal}(L_0/K)^{\varphi_{L/K}(u)}$, for $0\leq u\in\mathbb R$,
is the trivial group $<\text{id}_{L_0}>$, as $L_0/K$ is a finite unramified 
extension. Thus, for $\sigma\in\text{Gal}(L/K)_u$, by the commutativity
of the diagram (\ref{r-t-square}),  
\begin{equation*}
t_{L_0}^L(\varphi_{L/K}(u))(\sigma)=\sigma\mid_{L_0}=\text{id}_{L_0},
\end{equation*}
and in return
\begin{equation*}
\sigma\mapsto\left(\text{id}_{L_0},\varphi^{-m}\sigma\right)
=(\text{id}_{L_0},\sigma),
\end{equation*}
where $\sigma\mid_{L_0}=\text{id}_{L_0}=\varphi^m\mid_{L_0}$ for some
$0\leq m\in\mathbb Z$ satisfying $d\mid m$. As $L$ is fixed by $\varphi^d$,
$\varphi^{-m}\sigma=\sigma$, thus 
$(\text{id}_{L_0},\varphi^{-m}\sigma)=(\text{id}_{L_0},\sigma)$. Now,
the injectivity of the morphism given by eq. (\ref{directproduct2})
and defined by eq. (\ref{directproduct-definition2}) is clear from 
the commutative square eq. (\ref{r-t-square}) and by the injectivity
of the arrow given by eq. (\ref{directproduct}) and defined by 
eq. (\ref{directproduct-definition}). Thus
it suffices to prove that this morphism is a surjection, which follows
from the triviality of $\text{Gal}(L_0/K)^{\varphi_{L/K}(u)}$
for $0\leq u\in\mathbb R$, and from the equality
$\text{Gal}(L/K)_{u}=\text{Gal}(L/L_0)^{\varphi_{L/K}(u)}$. 
\end{proof}
Now, $L/L_0$ is an $APF$-Galois sub-extension of $L/K$ by part (i) of 
Lemma 3.3 in \cite{ikeda-serbest}. Let 
$\varphi_{L/L_0}:\mathbb R_{\geq-1}\rightarrow\mathbb R_{\geq-1}$ 
be the Hasse-Herbrand function corresponding to
the $APF$-extension $L/L_0$ defined by eq. (3.1) in \cite{ikeda-serbest}, 
which is piecewise-linear and continuous.
So there exists a unique number $w=w(u,L/K)\in\mathbb R_{\geq -1}$, depending
on $u$, satisfying
$\varphi_{L/K}(u)=\varphi_{L/L_0}(w)$ and
\begin{equation*}
\text{Gal}(L/L_0)^{\varphi_{L/K}(u)}=
\text{Gal}(L/L_0)^{\varphi_{L/L_0}(w)}=
\text{Gal}(L/L_0)_w .
\end{equation*}
Thus, Lemma \ref{directproduct-lemma} can be reformulated as follows. The
topological isomorphism defined by eq.s (\ref{directproduct}) and 
(\ref{directproduct-definition}) induces a topological isomorphism
\begin{equation*}
\text{Gal}(L/K)_u\simeq \left<\text{id}_{L_0}\right>\times
\text{Gal}(L/L_0)_{w(u,L/K)},
\end{equation*}
for every $0\leq u\in\mathbb R$.

For each $0\leq i\in\mathbb R$, consider the $i^{th}$ higher unit group
$U_{\widetilde{\mathbb X}(L/K)}^i$ of the field $\widetilde{\mathbb X}(L/K)$,
and define the group
\begin{equation}
\label{diamond-group-i} 
\left(U_{\widetilde{\mathbb X}{(L/K)}}^\diamond\right)^i=
U_{\widetilde{\mathbb X}{(L/K)}}^\diamond\cap U_{\widetilde{\mathbb X}(L/K)}^i.
\end{equation}
Now, Fesenko ramification theorem, stated as Theorem 5.8 in 
\cite{ikeda-serbest}, has the following generalization for the 
generalized arrow $\pmb{\phi}_{L/K}^{(\varphi)}$ corresponding to the 
extension $L/K$, which is an infinite $APF$-Galois sub-extension of 
$K_{\varphi^d}/K$ with residue-class degree $[\kappa_L:\kappa_K]=d$.
\begin{theorem}[Ramification theorem]
\label{ramification1} 
For $0\leq u\in\mathbb R$,
let $\text{Gal}(L/K)_u$ denote the $u^{th}$ higher ramification 
subgroup in the lower numbering of the Galois group $\text{Gal}(L/K)$ 
corresponding to the infinite $APF$-Galois sub-extension $L/K$ of 
$K_{\varphi^d}/K$ with residue-class degree $[\kappa_L:\kappa_K]=d$. 
Then, for $0\leq n\in\mathbb Z$, there exists the inclusion
\begin{multline*}
\pmb{\phi}_{L/K}^{(\varphi)}
\left(\text{Gal}(L/K)_{\psi_{L/K}\circ\varphi_{L/L_0}(n)}
-\text{Gal}(L/K)_{\psi_{L/K}\circ\varphi_{L/L_0}(n+1)}\right)
\subseteq \\
\left<1_{K^\times/N_{L_0/K}L_0^\times}\right>\times
\left(
\left(U_{\widetilde{\mathbb X}{(L/K)}}^\diamond\right)^{n} 
U_{\mathbb X(L/K)}/U_{\mathbb X(L/K)}-
\left(U_{\widetilde{\mathbb X}{(L/K)}}^\diamond\right)^{n+1} 
U_{\mathbb X(L/K)}/U_{\mathbb X(L/K)}
\right).
\end{multline*}
\end{theorem}
\begin{proof}
Now, we start with the following general observation.
Let $0\leq u\in\mathbb R$.
Let $\tau\in\text{Gal}(L/K)_u=\text{Gal}(L/L_0)^{\varphi_{L/K}(u)}$. 
Then, by the definition of the generalized
arrow $\pmb{\phi}_{L/K}^{(\varphi)}$ reformulated as in 
Remark \ref{generalized-arrow-definition-new},
\begin{equation*}
\pmb{\phi}_{L/K}^{(\varphi)}(\tau) =
\left(\pi_K^m.N_{L_0/K}L_0^\times,
\phi_{L/L_0}^{(\varphi')}(\varphi^{-m}\tau)\right),
\end{equation*}
where $\tau\mid_{L_0}=\varphi^m\mid_{L_0}$, for some $0\leq m\in\mathbb Z$ 
satisfying $d\mid m$, as $\tau\in\text{Gal}(L/K)_u$ and 
$\tau\mid_{L_0}=t_{L_0}^L(\varphi_{L/K}(u))(\tau)\in
\text{Gal}(L_0/K)^{\varphi_{L/K}(u)}=\left<\text{id}_{L_0}\right>$. Thus,
\begin{equation*}
\pmb{\phi}_{L/K}^{(\varphi)}(\tau) =
\left(1_{K^\times/N_{L_0/K}L_0^\times},
\phi_{L/L_0}^{(\varphi')}(\varphi^{-m}\tau)\right),
\end{equation*}
as $m=dm'$ and thereby $\pi_K^{dm'}N_{L_0/K}L_0^\times = N_{L_0/K}L_0^\times =
1_{K^\times/N_{L_0/K}L_0^\times}$, since $N_{L_0/K}\pi_K^{m'}=\pi_K^m$. 
Therefore,
\begin{equation*}
\pmb{\phi}_{L/K}^{(\varphi)}(\tau) = 
\left(1_{K^\times/N_{L_0/K}L_0^\times}, 
\phi_{L/L_0}^{(\varphi')}(\tau)\right),
\end{equation*}
since $\varphi^{-m}\tau=\tau$ in $\text{Gal}(L/L_0)$, as $d\mid m$, and
$L\subset K_{\varphi^d}$.

Now, to prove the theorem, let
$u=\psi_{L/K}\circ\varphi_{L/L_0}(n)$
and $u'=\psi_{L/K}\circ\varphi_{L/L_0}(n+1)$. Then, for
any $\tau\in\text{Gal}(L/K)_u-\text{Gal}(L/K)_{u'}$,
it follows from Fesenko ramification theorem 
(cf. Theorem 5.8 in \cite{ikeda-serbest}) that
the second coordinate of $\pmb{\phi}_{L/K}^{(\varphi)}(\tau)$
satisfies 
\begin{equation*}
\phi_{L/L_0}^{(\varphi')}(\tau)\in
\left(U_{\widetilde{\mathbb X}{(L/K)}}^\diamond\right)^{n}U_{\mathbb X(L/K)}
/U_{\mathbb X(L/K)}-
\left(U_{\widetilde{\mathbb X}{(L/K)}}^\diamond\right)^{n+1}U_{\mathbb X(L/K)}
/U_{\mathbb X(L/K)} ,
\end{equation*}
since 
\begin{equation*}
\text{Gal}(L/K)_u=\text{Gal}(L/L_0)^{\varphi_{L/K}(u)}
=\text{Gal}(L/L_0)^{\varphi_{L/L_0}(n)}=\text{Gal}(L/L_0)_n
\end{equation*} 
and likewise
\begin{equation*}
\text{Gal}(L/K)_{u'}=\text{Gal}(L/L_0)^{\varphi_{L/K}(u')}
=\text{Gal}(L/L_0)^{\varphi_{L/L_0}(n+1)}=\text{Gal}(L/L_0)_{n+1},
\end{equation*}
which completes the proof.
\end{proof}
Now, let $M/K$ be an infinite Galois sub-extension of $L/K$. Thus, by 
Lemma 3.3 of \cite{ikeda-serbest}, $M$ is an $APF$-Galois extension over $K$.
We further assume that, the residue-class degree $[\kappa_M:\kappa_K]=d'$ and 
$K\subset M\subset K_{\varphi^{d'}}$ for some $d'\mid d$.
Let
\begin{equation*}
\pmb{\phi}_{M/K}^{(\varphi)}:\text{Gal}(M/K)\rightarrow
K^\times/N_{M_0/K}M_0^\times
\times U_{\widetilde{\mathbb X}(M/K)}^\diamond/U_{\mathbb X(M/K)}
\end{equation*}
be the corresponding generalized arrow defined for the extension $M/K$.
Here, $M_0$ is defined by $M_0=M\cap K^{nr}=K_{d'}^{nr}$.

Now, let
\begin{equation*}
K\subset L_o=E_o\subset E_1\subset\cdots\subset E_i\subset\cdots\subset L
\end{equation*}
be an ascending chain satisfying $L=\bigcup_{0\leq i\in\mathbb Z}E_i$
and $[E_{i+1}:E_i]<\infty$ for every $0\leq i\in\mathbb Z$. Then
\begin{equation*}
K\subset M_o=E_o\cap M\subseteq E_1\cap M\subseteq\cdots\subseteq E_i\cap M
\subseteq\cdots\subset M
\end{equation*}
is an ascending chain of field extensions satisfying the conditions
$M=\bigcup_{0\leq i\in\mathbb Z}(E_i\cap M)$ and also 
$[E_{i+1}\cap M:E_i\cap M]<\infty$ for every $0\leq i\in\mathbb Z$.  
Thus, we construct $\mathbb X(M/K)$ by the sequence 
$(E_i\cap M)_{0\leq i\in\mathbb Z}$ and $\widetilde{\mathbb X}(M/K)$ by
the sequence $(\widetilde{E_i\cap M})_{0\leq i\in\mathbb Z}$.
Note that, $E_i\cap M\neq E_i$ for every $0\leq i\in\mathbb Z$.
Furthermore, the commutative square
\begin{equation*}
\SelectTips{cm}{}\xymatrix{
{\widetilde{E}_i^\times}
\ar[d]_{\prod_{0\leq\ell\lneq f(L/M)}(\varphi^{d'})^\ell
\widetilde{N}_{E_i/E_i\cap M}} & & {\widetilde{E}_{i'}^\times}
\ar[ll]_{\widetilde{N}_{E_{i'}/E_i}}
\ar[d]^{\prod_{0\leq\ell\lneq f(L/M)}(\varphi^{d'})^\ell
\widetilde{N}_{E_{i'}/E_{i'}\cap M}} \\
{\widetilde{E_i\cap M}^\times} & & {\widetilde{E_{i'}\cap M}^\times}
\ar[ll]_{\widetilde{N}_{E_{i'}\cap M/E_i\cap M}}
}
\end{equation*}
for every pair $0\leq i,i'\in\mathbb Z$ satisfying $i\leq i'$, induces the 
group homomorphism
\begin{equation}
\label{coleman-norm-map}
\widetilde{\mathcal N}_{L/M}=
\varprojlim_{0\leq i\in\mathbb Z}
\left(\prod_{0\leq\ell\lneq f(L/M)}(\varphi^{d'})^\ell
\widetilde{N}_{E_i/E_i\cap M}\right) : \widetilde{\mathbb X}(L/K)^\times
\rightarrow
\widetilde{\mathbb X}(M/K)^\times
\end{equation}
defined by
\begin{equation}
\label{coleman-norm-map-definition}
\widetilde{\mathcal N}_{L/M}
\left((\alpha_{\widetilde{E}_i})_{0\leq i\in\mathbb Z}\right)= 
\left(\prod_{0\leq\ell\lneq f(L/M)}(\varphi^{d'})^\ell
\widetilde{N}_{E_i/E_i\cap M}(\alpha_{\widetilde{E}_i})\right)
_{0\leq i\in\mathbb Z},
\end{equation}
for every $(\alpha_{\widetilde{E}_i})_{0\leq i\in\mathbb Z}
\in\widetilde{\mathbb X}(L/K)^\times$.
\begin{remark}
\label{chain-independence}
The group homomorphism
\begin{equation*}
\widetilde{\mathcal N}_{L/M} : \widetilde{\mathbb X}(L/K)^\times\rightarrow
\widetilde{\mathbb X}(M/K)^\times
\end{equation*}
defined by eq.s (\ref{coleman-norm-map}) and 
(\ref{coleman-norm-map-definition}) does \textit{not} depend on the choice 
of the ascending chain 
\begin{equation*}
K\subset L_o=E_o\subset E_1\subset\cdots\subset E_i\subset\cdots\subset L
\end{equation*}
satisfying $L=\bigcup_{0\leq i\in\mathbb Z}E_i$ and $[E_{i+1}:E_i]<\infty$ 
for every $0\leq i\in\mathbb Z$.
\end{remark}
\begin{remark}
\label{norm-norm}
For $0\leq i\in\mathbb Z$, let 
$E_{i,0}^{(E_i\cap M)}=E_i\cap (E_i\cap M)^{nr}$ be the maximal unramified
extension of $E_i\cap M$ inside $E_i$. To simplify the notation, let
$E_{i,0}=E_i\cap (E_i\cap M)^{nr}$. Then the Galois group 
$\text{Gal}(E_{i,0}/E_i\cap M)$ is cyclic of order $f(L/M)=\frac{d}{d'}$ 
generated by $\varphi^{d'}$. Thus, for $\alpha\in E_i$,
\begin{equation*}
N_{E_i/E_i\cap M}(\alpha)=\widetilde{N}_{E_i/E_i\cap M}(\alpha)
^{1+\varphi^{d'}+\cdots+\varphi^{d'(f(L/M)-1)}}.
\end{equation*}
\end{remark}
The basic properties of this group homomorphism are the following.
\begin{itemize}
\item[(i)]
If $U=(u_{\widetilde{E}_i})_{0\leq i\in\mathbb Z}
\in U_{\widetilde{\mathbb X}(L/K)}$, then
$\widetilde{\mathcal N}_{L/M}(U)\in U_{\widetilde{\mathbb X}(M/K)}$.
\begin{proof}
In fact, following the definition of the valuation 
$\nu_{\widetilde{\mathbb X}(M/K)}$ of $\widetilde{\mathbb X}(M/K)$ and
the definition of the valuation $\nu_{\widetilde{\mathbb X}(L/K)}$ of
$\widetilde{\mathbb X}(L/K)$, it follows that
\begin{equation*}
\begin{aligned}
\nu_{\widetilde{\mathbb X}(M/K)}
\left(\widetilde{\mathcal N}_{L/M}(U)\right) & =
\nu_{\widetilde{\mathbb X}(M/K)}
\left(\left(\prod_{0\leq\ell\lneq f(L/M)}(\varphi^{d'})^\ell
\widetilde{N}_{E_i/E_i\cap M}(u_{\widetilde{E}_i})\right)
_{0\leq i\in\mathbb Z}\right) \\
& = \nu_{\widetilde{K}}\left(\prod_{0\leq\ell\lneq f(L/M)}(\varphi^{d'})
^\ell (u_{\widetilde{K}})\right) \\
& = \sum_{0\leq\ell\lneq f(L/M)}\nu_{\widetilde{K}}\left((\varphi^{d'})^\ell 
(u_{\widetilde{K}})\right) \\
& = \sum_{0\leq\ell\lneq f(L/M)}\nu_{\widetilde{K}}(u_{\widetilde{K}}) \\
& = 0,
\end{aligned}
\end{equation*}
as 
$\nu_{\widetilde{K}}\left((\varphi^{d'})^\ell(u_{\widetilde{K}})\right)=
\nu_{\widetilde{K}}(u_{\widetilde{K}})$ for 
$\ell = 1,\cdots,f(L/M)-1$, and
\begin{equation*}
\nu_{\widetilde{\mathbb X}(L/K)}(U)=
\nu_{\widetilde{K}}(u_{\widetilde{K}})=0,
\end{equation*}
since $U\in U_{\widetilde{\mathbb X}(L/K)}$. 
\end{proof}
\item[(ii)]
If $U=(u_{\widetilde{E}_i})_{0\leq i\in\mathbb Z}
\in U_{\widetilde{\mathbb X}(L/K)}^\diamond$, 
then $\widetilde{\mathcal N}_{L/M}(U)
\in U_{\widetilde{\mathbb X}(M/K)}^\diamond$.
\begin{proof}
Note that, $\widetilde{L}_0=\widetilde{K}$ and 
$\widetilde{M}_0=\widetilde{K}$. Now, the assertion follows by 
observing that
\begin{equation*}
\text{Pr}_{\widetilde{K}}(U)=u_{\widetilde{K}}\in U_{L_0}
\end{equation*} 
and
\begin{equation*} 
\begin{aligned}
\text{Pr}_{\widetilde{K}}\left(\widetilde{\mathcal N}_{L/M}(U)\right)&=
\prod_{0\leq\ell\lneq f(L/M)}(\varphi^{d'})^\ell
\widetilde{N}_{E_o/E_o\cap M}(u_{\widetilde{E}_o})\\
&=\prod_{0\leq\ell\lneq f(L/M)}(\varphi^{d'})^\ell u_{\widetilde{K}}\\
&=N_{E_o/E_o\cap M}(u_{\widetilde{K}})\in U_{M_0}.
\end{aligned}
\end{equation*} 
\end{proof}
\item[(iii)]
If $U=(u_{E_i})_{0\leq i\in\mathbb Z}\in U_{{\mathbb X}(L/K)}$, then 
$\widetilde{\mathcal N}_{L/M}(U)\in U_{{\mathbb X}(M/K)}$.
\begin{proof}
The assertion follows by the definition 
eq. (\ref{coleman-norm-map-definition}) of the 
homomorphism eq. (\ref{coleman-norm-map}) combined with the fact that 
$\widetilde{N}_{E_i/E_i\cap M}(u_{E_i})^{1+\varphi^{d'}+\cdots
+\varphi^{d'(f(L/M)-1)}}=N_{E_i/E_i\cap M}(u_{E_i})$
for every $u_{E_i}\in U_{E_i}$ and for every $0\leq i\in\mathbb Z$ by Remark
\ref{norm-norm}.
\end{proof}
\end{itemize}
Note that,
$\widetilde{N}_{E_i/E_{i-1}}\left(\alpha^{1+\varphi^{d'}+\cdots+
\varphi^{d'(f-1)}}\right)=
\widetilde{N}_{E_i/E_{i-1}}\left(\alpha\right)^{1+\varphi^{d'}+\cdots+
\varphi^{d'(f-1)}}$, 
for any $\alpha\in\widetilde{E}_i$ with $1\leq i\in\mathbb Z$,
where $f=f(L/M)$. Thus, there exists a homomorphism
\begin{equation}
\label{diamond-frobenius}
\left<\varphi\right>_{L/M}:\widetilde{\mathbb X}(L/L_0)^\times\rightarrow
\widetilde{\mathbb X}(L/L_0)^\times
\end{equation}
defined by
\begin{equation}
\label{diamond-frobenius-definition}
\left<\varphi\right>_{L/M}:\left(\alpha_{\widetilde{E}_i}\right)
_{0\leq i\in\mathbb Z}\mapsto\left(\alpha_{\widetilde{E}_i}^{1+\varphi^{d'}
+\cdots+\varphi^{d'(f-1)}}\right)_{0\leq i\in\mathbb Z},
\end{equation}
for every $\left(\alpha_{\widetilde{E}_i}\right)_{0\leq i\in\mathbb Z}\in
\widetilde{\mathbb X}(L/L_0)^\times$. The basic properties of this group 
homomorphism are the following.
\begin{itemize}
\item[(i)] 
$\left<\varphi\right>_{L/M}\left(U_{\widetilde{\mathbb X}(L/L_0)}\right)
\subseteq U_{\widetilde{\mathbb X}(L/L_0)}$.
\begin{proof}
In fact, following the definition of the valuation 
 $\nu_{\widetilde{\mathbb X}(L/L_0)}$ of $\widetilde{\mathbb X}(L/L_0)$, 
it follows that
\begin{equation*}
\begin{aligned}
\nu_{\widetilde{\mathbb X}(L/L_0)}
\left(\left<\varphi\right>_{L/M}(U)\right) & =
\nu_{\widetilde{\mathbb X}(L/L_0)}
\left(
u_{\widetilde{E}_i}^{1+\varphi^{d'}+\cdots+\varphi^{d'(f(L/M)-1)}}\right)
_{0\leq i\in\mathbb Z} \\
& = \nu_{\widetilde{K}}\left(u_{\widetilde{K}}^{1+\varphi^{d'}+\cdots
+\varphi^{d'(f(L/M)-1)}}\right) \\
& = \sum_{0\leq\ell\lneq f(L/M)}\nu_{\widetilde{K}}\left((\varphi^{d'})^\ell 
(u_{\widetilde{K}})\right) \\
& = \sum_{0\leq\ell\lneq f(L/M)}\nu_{\widetilde{K}}(u_{\widetilde{K}}) \\
& = 0,
\end{aligned}
\end{equation*}
as 
$\nu_{\widetilde{K}}\left((\varphi^{d'})^\ell(u_{\widetilde{K}})\right)=
\nu_{\widetilde{K}}(u_{\widetilde{K}})$ for 
$\ell = 1,\cdots,f(L/M)-1$, and
\begin{equation*}
\nu_{\widetilde{\mathbb X}(L/K)}(U)=
\nu_{\widetilde{K}}(u_{\widetilde{K}})=0,
\end{equation*}
since $U\in U_{\widetilde{\mathbb X}(L/K)}$. 
\end{proof}
\item[(ii)]
$\left<\varphi\right>_{L/M}\left(U^\diamond_{\widetilde{\mathbb X}(L/L_0)}
\right)\subseteq U^\diamond_{\widetilde{\mathbb X}(L/L_0)}$.
\begin{proof}
Note that, $\widetilde{L}_0=\widetilde{K}$. Now, the assertion follows by 
observing that
\begin{equation*}
\text{Pr}_{\widetilde{K}}(U)=u_{\widetilde{K}}\in U_{L_0}
\end{equation*} 
and
\begin{equation*} 
\begin{aligned}
\text{Pr}_{\widetilde{K}}\left(\left<\varphi\right>_{L/M}(U)\right)&=
u_{\widetilde{K}}^{1+\varphi^{d'}+\cdots+\varphi^{d'(f(L/M)-1)}}\\
&=\prod_{0\leq\ell\lneq f(L/M)}(\varphi^{d'})^\ell u_{\widetilde{K}}\\
&=N_{E_o/E_o\cap M}(u_{\widetilde{K}})\in U_{M_0}\subseteq U_{L_0}.
\end{aligned}
\end{equation*} 
\end{proof}
\item[(iii)]
$\left<\varphi\right>_{L/M}\left(U_{{\mathbb X}(L/L_0)}\right)\subseteq 
U_{{\mathbb X}(L/L_0)}$.
\begin{proof}
Clearly, for $U=(u_{E_i})_{0\leq i\in\mathbb Z}\in U_{\mathbb X(L/L_o)}$,
\begin{equation*}
\left<\varphi\right>_{L/M}(U)=(u_{E_i}^{1+\varphi^{d'}+\cdots+
\varphi^{d'(f(L/M)-1)}})_{0\leq i\in\mathbb Z}\in
U_{\mathbb X(L/L_o)},
\end{equation*}
as $u_{E_i}\in U_{E_i}$ for every $0\leq i\in\mathbb Z$.
\end{proof}
\end{itemize}
Thus, there exists a group homomorphism 
\begin{equation}
\label{coleman-norm-0}
\widetilde{\mathcal N}_{L/M}\circ\left<\varphi\right>_{L/M}:
\widetilde{\mathbb X}(L/K)^\times\rightarrow\widetilde{\mathbb X}(M/K)^\times
\end{equation}
satisfying
\begin{itemize}
\item[(i)]
$\widetilde{\mathcal N}_{L/M}\circ\left<\varphi\right>_{L/M}
\left(U_{\widetilde{\mathbb X}(L/K)}\right)
\subseteq U_{\widetilde{\mathbb X}(M/K)}$;
\item[(ii)]
$\widetilde{\mathcal N}_{L/M}\circ\left<\varphi\right>_{L/M}
\left(U^\diamond_{\widetilde{\mathbb X}(L/K)}\right)
\subseteq U^\diamond_{\widetilde{\mathbb X}(M/K)}$;
\item[(iii)]
$\widetilde{\mathcal N}_{L/M}\circ\left<\varphi\right>_{L/M}
\left(U_{{\mathbb X}(L/K)}\right)
\subseteq U_{{\mathbb X}(M/K)}$.
\end{itemize}
Now, define the \textit{Coleman norm map}
\begin{equation}
\label{coleman-norm-1}
\widetilde{\mathcal N}_{L/M}^{\text{Coleman}} : 
U_{\widetilde{\mathbb X}(L/K)}^\diamond/U_{\mathbb X(L/K)}\rightarrow 
U_{\widetilde{\mathbb X}(M/K)}^\diamond/U_{\mathbb X(M/K)}
\end{equation}
from $L$ to $M$ by
\begin{equation}
\label{coleman-norm-1-def}
\widetilde{\mathcal N}_{L/M}^{\text{Coleman}}(\overline{U}) =
\widetilde{\mathcal N}_{L/M}\circ\left<\varphi\right>_{L/M}(U). 
U_{\mathbb X(M/K)},
\end{equation}
for every $U\in U_{\widetilde{\mathbb X}(L/K)}^\diamond$, where $\overline{U}$
denotes, as before, the coset $U.U_{\mathbb X(L/K)}$ in 
$U_{\widetilde{\mathbb X}(L/K)}^\diamond/U_{\mathbb X(L/K)}$.
\begin{lemma}
\label{square1-preliminary}
For an infinite Galois sub-extension $M/K$ of $L/K$ such that the 
residue-class degree $[\kappa_M:\kappa_K]=d'$ and 
$K\subset M\subset K_{\varphi^{d'}}$ for some $d'\mid d$,
the square
\begin{equation}
\label{square-L/M-U}
\SelectTips{cm}{}\xymatrix{
{\text{Gal}(L/L_0)}\ar[r]^-{\phi_{L/L_0}^{(\varphi^d)}}
\ar[d]_{\text{res}_M} & 
{U_{\widetilde{\mathbb X}{(L/L_0)}}^\diamond/U_{\mathbb X(L/L_0)}}
\ar[d]^{\widetilde{\mathcal N}_{L/M}^{\text{Coleman}}} \\
{\text{Gal}(M/M_0)}\ar[r]^-{\phi_{M/M_0}^{(\varphi^{d'})}} & 
{U_{\widetilde{\mathbb X}{(M/M_0)}}^\diamond/U_{\mathbb X(M/M_0)}},
}
\end{equation}
where the right-vertical arrow is the Coleman norm map 
$\widetilde{\mathcal N}_{L/M}^{\text{Coleman}}$ from $L$ to $M$ 
defined by eq.s 
(\ref{coleman-norm-1}) and (\ref{coleman-norm-1-def}), is commutative.
\end{lemma} 
\begin{proof}
For $\sigma\in\text{Gal}(L/L_0)$, 
$\text{res}_M(\sigma)=\sigma\mid_M\in\text{Gal}(M/M_0)$, 
as $L_0\cap M=L\cap K^{nr}\cap M=M\cap K^{nr}=M_0$.
Now, for any $\sigma\in\text{Gal}(L/L_0)$, following the definition
$\phi_{L/L_0}^{(\varphi^d)}(\sigma)=U_\sigma . U_{\mathbb X(L/L_0)}$, 
where $U_\sigma\in U_{\widetilde{\mathbb X}(L/L_0)}^\diamond$ satisfies
the equation $U_\sigma^{1-\varphi^d}=\Pi_{\varphi^d ; L/L_0}^{\sigma-1}$.
Thus, to prove the commutativity of the square, it suffices to prove that
\begin{equation*}
\widetilde{\mathcal N}_{L/M}(U_\sigma^{1+\varphi^{d'}+\cdots+\varphi
^{d'(f(L/M)-1)}})\equiv U_{\sigma\mid_M}\pmod{U_{\mathbb X(M/M_0)}},
\end{equation*}
where $U_{\sigma\mid_M}\in U_{\widetilde{\mathbb X}(M/M_0)}^\diamond$ satisfies
the equation $U_{\sigma\mid_M}^{1-\varphi^{d'}}
=\Pi_{\varphi^{d'} ; M/M_0}^{\sigma\mid_M-1}$. Now, without loss of
generality, in view of Remark \ref{chain-independence}, fix a 
\textit{basic sequence} (cf. \cite{ikeda-serbest})
\begin{equation*}
L_0=E_0\subset E_1\subset\cdots\subset E_i\subset\cdots\subset L ,
\end{equation*}
where 
\begin{itemize}
\item[(i)]
$L=\bigcup_{0\leq i\in\mathbb Z}E_i$;
\item[(ii)]
$E_i/L_0$ is a Galois extension for every $0\leq i\in\mathbb Z$;
\item[(iii)]
$E_{i+1}/E_i$ is cyclic of prime degree 
$[E_{i+1}:E_i]=p=\text{char}(\kappa_{L_0})$ for each $1\leq i\in\mathbb Z$;
\item[(iv)]
$E_1/E_0$ is cyclic of degree relatively prime to $p$. 
\end{itemize}
Thus, each extension $E_i/L_0$ is finite and Galois for $0\leq i\in\mathbb Z$.
Now, note that
\begin{equation*}
\widetilde{\mathcal N}_{L/M}(U_\sigma^{1+\varphi^{d'}+\cdots+
\varphi^{d'(f(L/M)-1)}})^{1-\varphi^{d'}} 
=\widetilde{\mathcal N}_{L/M}(U_\sigma)^{1-\varphi^d}
=\widetilde{\mathcal N}_{L/M}(U_\sigma^{1-\varphi^d}).
\end{equation*}
As $U_\sigma^{1-\varphi^d}=\Pi_{\varphi^d ; L/L_0}^{\sigma-1}$, 
setting $U_\sigma=(u_{\widetilde{E}_i})_{0\leq i\in\mathbb Z}$, for 
$0\leq i\in\mathbb Z$,
\begin{equation*}
\begin{aligned}
\widetilde{\mathcal N}_{L/M}(U_\sigma^{1+\varphi^{d'}+\cdots+
\varphi^{d'(f(L/M)-1)}})_i^{1-\varphi^{d'}}& =
\widetilde{\mathcal N}_{L/M}\left(\Pi_{\varphi^d ; L/L_0}^{\sigma-1}\right)_i\\
&=\widetilde{N}_{E_i/E_i\cap M}(\pi_{E_i}^{\sigma-1})^{1+\varphi^{d'}+\cdots+
\varphi^{d'(f(L/M)-1)}}\\
&=N_{E_i/E_i\cap M}(\pi_{E_i}^{\sigma-1})\\
&=\pi_{E_i\cap M}^{\sigma\mid_{M}-1} .
\end{aligned}
\end{equation*}
Now, it follows that, 
$\widetilde{\mathcal N}_{L/M}\circ\left<\varphi\right>_{L/M}(U_\sigma)
^{1-\varphi^{d'}}=\Pi_{\varphi^{d'};M/M_0}^{\sigma\mid_M-1}$, 
which yields the congruence
$\widetilde{\mathcal N}_{L/M}\circ\left<\varphi\right>_{L/M}(U_\sigma)
\equiv U_{\sigma\mid_M}\pmod{U_{\mathbb X(M/M_0)}}$ completing the proof.
\end{proof}
So, we have the following theorem.
\begin{theorem}
\label{square1}
For an infinite Galois sub-extension $M/K$ of $L/K$ such that the 
residue-class degree $[\kappa_M:\kappa_K]=d'$ and 
$K\subset M\subset K_{\varphi^{d'}}$ for some $d'\mid d$,
the square
\begin{equation*}
\SelectTips{cm}{}\xymatrix{
{\text{Gal}(L/K)}\ar[r]^-{\pmb{\phi}_{L/K}^{(\varphi)}}\ar[d]_{\text{res}_M} & 
{K^\times/N_{L_0/K}L_0^\times
\times
U_{\widetilde{\mathbb X}{(L/K)}}^\diamond
/U_{\mathbb X(L/K)}}\ar[d]^{\left(e_{L_0/M_0}^{\text{CFT}},
\widetilde{\mathcal N}_{L/M}^{\text{Coleman}}\right)} \\
{\text{Gal}(M/K)}\ar[r]^-{\pmb{\phi}_{M/K}^{(\varphi)}} & 
{K^\times/N_{M_0/K}M_0^\times
\times
U_{\widetilde{\mathbb X}{(M/K)}}^\diamond
/U_{\mathbb X(M/K)}},
}
\end{equation*}
where the right-vertical arrow
\begin{equation*}
K^\times/N_{L_0/K}L_0^\times
\times U_{\widetilde{\mathbb X}{(L/K)}}^\diamond/U_{\mathbb X(L/K)}
\xrightarrow{\left(e_{L_0/M_0}^{\text{CFT}},\widetilde{\mathcal N}_{L/M}
^{\text{Coleman}}\right)} 
K^\times/N_{M_0/K}M_0^\times
\times U_{\widetilde{\mathbb X}{(M/K)}}^\diamond/U_{\mathbb X(M/K)}
\end{equation*}
defined by
\begin{equation*}
\left(e_{L_0/M_0}^{\text{CFT}},\widetilde{\mathcal N}_{L/M}
^{\text{Coleman}}\right) :
(\overline{a},\overline{U})\mapsto\left(e_{L_0/M_0}^{\text{CFT}}(\overline{a}),
\widetilde{\mathcal N}_{L/M}^{\text{Coleman}}(\overline{U})\right)
\end{equation*}
for every $(\overline{a},\overline{U})\in K^\times/N_{L_0/K}L_0^\times\times
U_{\widetilde{\mathbb X}(L/K)}^\diamond/U_{\mathbb X(L/K)}$, is commutative.
Here, 
\begin{equation*}
e_{L_0/M_0}^{\text{CFT}}:K^\times/N_{L_0/K}L_0^\times\rightarrow
K^\times/N_{M_0/K}M_0^\times
\end{equation*} 
is the natural inclusion defined via the existence theorem of local 
class field theory.
\end{theorem}
\begin{proof}
By the isomorphism defined by eq.s (\ref{directproduct}) and 
(\ref{directproduct-definition}), for $\sigma\in\text{Gal}(L/K)$, there
exists a unique $0\leq m\in\mathbb Z$ such that $\sigma\mid_{L_0}=\varphi^m$
and $\varphi^{-m}\sigma\in\text{Gal}(L/L_0)$. Now, following the definition,
\begin{equation*}
\pmb{\phi}_{L/K}^{(\varphi)}(\sigma) = \left(\pi_K^m N_{L_0/K}L_0^\times ,
\phi_{L/L_0}^{(\varphi^d)}(\varphi^{-m}\sigma)\right).
\end{equation*}
Thus,
\begin{equation*}
\begin{aligned}
\left(e_{L_0/M_0}^{\text{CFT}},
\widetilde{\mathcal N}_{L/M}^{\text{Coleman}}\right)
\left(\pi_K^m N_{L_0/K}L_0^\times , 
\phi_{L/L_0}^{(\varphi^d)}(\varphi^{-m}\sigma)\right)  =\\ 
\left(e_{L_0/M_0}^{\text{CFT}}(\pi_K^m N_{L_0/K}L_0^\times),
\widetilde{\mathcal N}_{L/M}^{\text{Coleman}}(\phi_{L/L_0}
^{(\varphi^d)}(\varphi^{-m}\sigma))\right)  =\\
\left(\pi_K^m N_{M_0/K}M_0^\times,\phi_{M/M_0}^{(\varphi^{d'})}
(\varphi^{-m}\sigma\mid_{M})\right) 
\end{aligned}
\end{equation*}
by Lemma \ref{square1-preliminary}.
Note that, by the existence theorem of local class field theory,
\begin{equation*}
e_{L_0/M_0}^{\text{CFT}}(\pi_K^m N_{L_0/K}L_0^\times)
=\pi_K^mN_{M_0/K}M_0^\times=\pi_K^{m'}N_{M_0/K}M_0^\times ,
\end{equation*}
where $0\leq m'\in\mathbb Z$ is the unique integer satisfying
$(\sigma\mid_M)\mid_{M_0}=\sigma\mid_{M_0}=\varphi^{m'}$ and 
$\varphi^{-m'}(\sigma\mid_{M})\in\text{Gal}(M/M_0)$.
Hence,
\begin{equation*}
\begin{aligned}
\left(e_{L_0/M_0}^{\text{CFT}},
\widetilde{\mathcal N}_{L/M}^{\text{Coleman}}\right)
(\pmb{\phi}_{L/K}^{(\varphi)}(\sigma)) &= 
\left(\pi_K^{m'}N_{M_0/K}M_0^\times ,\phi_{M/M_0}^{(\varphi^{d'})}
(\varphi^{-m}\sigma\mid_{M})\right)\\
& = \left(\pi_K^{m'}N_{M_0/K}M_0^\times ,\phi_{M/M_0}^{(\varphi^{d'})}
(\varphi^{-m'}(\sigma\mid_{M}))\right)\\
&=\pmb{\phi}_{M/K}^{(\varphi)}(\text{res}_M(\sigma))
\end{aligned}
\end{equation*}
by Remark \ref{directproduct-diagram} part (i), which completes the proof.
\end{proof}
Now, let $F/K$ be a finite sub-extension of $L/K$. Thus, $L/F$ is an infinite
$APF$-Galois extension (cf. Lemma 3.3 of \cite{ikeda-serbest}). Fix a 
Lubin-Tate splitting $\varphi_F$ over $F$. Now, assume that the residue-class
degree $[\kappa_L : \kappa_F]=d'$, for some $d'\mid d$, 
and there exists the chain of field extensions 
\begin{equation*}
F\subset L\subset F_{(\varphi_F)^{d'}}.
\end{equation*}
Thus, there exists the generalized arrow
\begin{equation*}
\pmb{\phi}_{L/F}^{(\varphi_F)}:\text{Gal}(L/F)\rightarrow
F^\times/N_{L_0^{(F)}/F}{L_0^{(F)}}^\times
\times
U_{\widetilde{\mathbb X}{(L/F)}}^\diamond
/U_{\mathbb X(L/F)}
\end{equation*}
corresponding to the extension $L/F$. Here, $L_0^{(F)}$ is defined as
usual by $L_0^{(F)}=L\cap F^{nr}=F^{nr}_{d'}$. Note that, there is the
following diagram of field extensions.
\begin{equation*}
\xymatrix{
{} & L\ar@{-}[d] & {} \\
{} & {L_0^{(F)}}\ar@{-}[dl]_{\text{totally-ramified}}
\ar@{-}[dr]^{[L_0^{(F)}:F]<\infty} & {} \\
{L_0^{(K)}}\ar@{-}[dr]_{[L_0^{(K)}:K]<\infty} & {} & 
F\ar@{-}[dl]^{[F:K]<\infty} \\
{} & K & {}
}
\end{equation*}
Thus, $L/L_0^{(F)}$ and $L/L_0^{(K)}$ are infinite totally-ramified
$APF$-Galois extensions, by Lemma 3.3 of \cite{ikeda-serbest}, and satisfy
$L_0^{(F)}\subset L\subset \left(L_0^{(F)}\right)_{\varphi_F^{d'}}$ and
$L_0^{(K)}\subset L\subset \left(L_0^{(K)}\right)_{\varphi_K^{d}}$.
\begin{remark}
\label{varphi-varphi}
Note that $L_0^{(F)}$ is compatible with $(L_0^{(K)},\varphi_{L_0^{(K)}})$,
in the sense of \cite{koch-deshalit} pp. 89, where 
$\varphi_{L_0^{(K)}}=\varphi_K^d$. Thus, $\varphi_{L_0^{(F)}}=\varphi_F^{d'}
=\varphi_{L_0^{(K)}}^{f(L_0^{(F)}/L_0^{(K)})}=\varphi_K^d$, as 
$L_0^{(F)}/L_0^{(K)}$ is totally-ramified.
\end{remark}
For the extension $L/L_0^{(F)}$, fix an ascending chain
\begin{equation*}
L_0^{(F)}=F_o\subset F_1\subset\cdots\subset F_i\subset\cdots\subset L
\end{equation*}
satisfying
$L=\bigcup_{0\leq i\in\mathbb Z}F_i$ and $[F_{i+1}:F_i]<\infty$ for
every $0\leq i\in\mathbb Z$. Following \cite{ikeda-serbest}, introduce 
the homomorphism
\begin{equation}
\label{Lambda-map}
\Lambda_{F/K} : \widetilde{\mathbb X}(L/L_0^{(F)})^\times\rightarrow
\widetilde{\mathbb X}(L/L_0^{(K)})^\times
\end{equation}
by
\begin{multline}
\label{Lambda-map-definition}
\Lambda_{F/K} :
(\alpha_{F_0}\xleftarrow{\widetilde{N}_{F_1/F_0}}\alpha_{F_1}
\xleftarrow{\widetilde{N}_{F_2/F_1}}\cdots)\mapsto \\
(\widetilde{N}_{L_0^{(F)}/L_0^{(K)}}(\alpha_{F_0})\xleftarrow
{\widetilde{N}_{L_0^{(F)}/L_0^{(K)}}}
\alpha_{F_0}\xleftarrow{\widetilde{N}_{F_1/F}}\alpha_{F_1}
\xleftarrow{\widetilde{N}_{F_2/F_1}}\cdots),
\end{multline}
for each $\left(\alpha_{F_i}\right)_{0\leq i\in\mathbb Z}
\in\widetilde{\mathbb X}(L/L_0^{(F)})^\times$. 
This homomorphism induces a group homomorphism
\begin{equation}
\label{lambda-map}
\lambda_{F/K} : 
U_{\widetilde{\mathbb X}{(L/L_0^{(F)})}}^\diamond/U_{\mathbb X(L/L_0^{(F)})}
\rightarrow
U_{\widetilde{\mathbb X}{(L/L_0^{(K)})}}^\diamond/U_{\mathbb X(L/L_0^{(K)})}
\end{equation}
defined by
\begin{equation}
\label{lambda-map-definition}
\lambda_{F/K}:\overline{U}\mapsto\Lambda_{F/K}(U).U_{\mathbb X(L/L_0^{(K)})},
\end{equation}
for every $U\in U_{\widetilde{\mathbb X}{(L/L_0^{(F)})}}^\diamond$, where
$\overline{U}$ denotes the coset $U.U_{\mathbb X(L/L_0^{(F)})}$ 
in $U_{\widetilde{\mathbb X}{(L/L_0^{(F)})}}^\diamond
/U_{\mathbb X(L/L_0^{(F)})}$ (for details cf. \cite{ikeda-serbest}).
\begin{lemma}
\label{square2-preliminary}
Let $F/K$ be a finite sub-extension of $L/K$. Fix a 
Lubin-Tate splitting $\varphi_F$ over $F$. Assume that the residue-class
degree $[\kappa_L : \kappa_F]=d'$ 
and $F\subset L\subset F_{(\varphi_F)^{d'}}$ for some $d'\mid d$. 
Then the square
\begin{equation}
\label{square-F/K-U-preliminary}
\SelectTips{cm}{}\xymatrix{
{\text{Gal}(L/L_0^{(F)})}\ar[r]^-{\phi_{L/L_0^{(F)}}^{(\varphi_K^{d})}}
\ar[d]_{\text{inc.}} & 
{U_{\widetilde{\mathbb X}{(L/L_0^{(F)})}}^\diamond/U_{\mathbb X(L/L_0^{(F)})}}
\ar[d]^{\lambda_{F/K}} \\
{\text{Gal}(L/L_0^{(K)})}\ar[r]^-{\phi_{L/L_0^{(K)}}^{(\varphi_K^d)}} & 
{U_{\widetilde{\mathbb X}{(L/L_0^{(K)})}}^\diamond/U_{\mathbb X(L/L_0^{(K)})}},
}
\end{equation}
where the right-vertical arrow
\begin{equation*}
\lambda_{F/K} : U_{\widetilde{\mathbb X}{(L/L_0^{(F)})}}^\diamond/
U_{\mathbb X(L/L_0^{(F)})}
\rightarrow U_{\widetilde{\mathbb X}{(L/L_0^{(K)})}}^\diamond/
U_{\mathbb X(L/L_0^{(K)})}
\end{equation*}
is defined by eq.s (\ref{lambda-map}) and
(\ref{lambda-map-definition}), 
is commutative.
\end{lemma}
\begin{proof}
Look at the proof of Theorem 5.12 of \cite{ikeda-serbest}.
\end{proof}
So, we have the following theorem.
\begin{theorem}
\label{square2}
Let $F/K$ be a finite sub-extension of $L/K$. Fix a 
Lubin-Tate splitting $\varphi_F$ over $F$. Assume that the residue-class
degree $[\kappa_L : \kappa_F]=d'$ 
and $F\subset L\subset F_{(\varphi_F)^{d'}}$ for some $d'\mid d$. 
Then the square
\begin{equation}
\label{square-F/K-U}
\SelectTips{cm}{}\xymatrix{
{\text{Gal}(L/F)}\ar[r]^-{\pmb{\phi}_{L/F}^{(\varphi_F)}}
\ar[d]_{\text{inc.}} & 
{F^\times/N_{L_0^{(F)}/F}{L_0^{(F)}}^\times
\times
U_{\widetilde{\mathbb X}{(L/F)}}^\diamond
/U_{\mathbb X(L/F)}}\ar[d]^{(N_{F/K},\lambda_{F/K})} \\
{\text{Gal}(L/K)}\ar[r]^-{\pmb{\phi}_{L/K}^{(\varphi_K)}} & 
{K^\times/N_{L_0^{(K)}/K}{L_0^{(K)}}^\times
\times
U_{\widetilde{\mathbb X}{(L/K)}}^\diamond
/U_{\mathbb X(L/K)}},
}
\end{equation}
where the right-vertical arrow
\begin{multline*}
(N_{F/K},\lambda_{F/K}) :
F^\times/N_{L_0^{(F)}/F}{L_0^{(F)}}^\times
\times
U_{\widetilde{\mathbb X}(L/F)}^\diamond
/U_{\mathbb X(L/F)}
\rightarrow \\
K^\times/N_{L_0^{(K)}/K}{L_0^{(K)}}^\times
\times
U_{\widetilde{\mathbb X}{(L/K)}}^\diamond
/U_{\mathbb X(L/K)}
\end{multline*}
defined by
\begin{equation*}
(N_{F/K},\lambda_{F/K}) :
(\overline{a},\overline{U})\mapsto\left(\overline{N_{F/K}(a)},
\lambda_{F/K}(\overline{U})\right),
\end{equation*}
for every $(\overline{a},\overline{U})\in 
{F^\times/N_{L_0^{(F)}/F}{L_0^{(F)}}^\times}\times
U_{\widetilde{\mathbb X}(L/F)}^\diamond/U_{\mathbb X(L/F)}$, 
is commutative.
\end{theorem}
\begin{proof}
Let $\sigma\in\text{Gal}(L/F)$. There exists $0\leq m\in\mathbb Z$ such that
$\sigma\mid_{L_0^{(F)}}=\varphi_F^m$ and $\varphi_F^{-m}\sigma\in
\text{Gal}(L/L_0^{(F)})$. Now,
\begin{equation*}
\pmb{\phi}_{L/F}^{(\varphi_F)}(\sigma)=
\left(\pi_F^m.N_{L_0^{(F)}/F}{L_0^{(F)}}^\times,\phi_{L/L_0^{(F)}}
^{(\varphi_K^d)}(\varphi_F^{-m}\sigma)\right)
\end{equation*}
and
\begin{equation*}
(N_{F/K},\lambda_{F/K})(\pmb{\phi}_{L/F}^{(\varphi_F)}(\sigma))=
\left(\pi_K^m.N_{L_0^{(K)}/K}{L_0^{(K)}}^\times,\phi_{L/L_0^{(K)}}
^{(\varphi_K^d)}(\varphi_F^{-m}\sigma)\right)
\end{equation*}
by the norm-compatibility of primes in the fixed Lubin-Tate labelling and by
Lemma \ref{square2-preliminary}. Now, there exists $0\leq m'\in\mathbb Z$ 
such that $\sigma\mid_{L_0^{(K)}}=\varphi_K^{m'}$ and 
$\varphi_K^{-m'}\sigma\in\text{Gal}(L/L_0^{(K)})$. 
By Remark \ref{directproduct-diagram} part (ii), it follows that
$\varphi_F^m\mid_{L_0^{(K)}}=\varphi_K^{m'}$ and $\varphi_F^{-m}\sigma=
\varphi_K^{-m'}\sigma$. By abelian local class field theory,
$N_{F/K} : \pi_F^m N_{L_0^{(F)}/F}{L_0^{(F)}}^\times\mapsto
\pi_K^{m'}N_{L_0^{(K)}/K}{L_0^{(K)}}^\times
=\pi_K^m.N_{L_0^{(K)}/K}{L_0^{(K)}}^\times $. Thus,
\begin{equation*}
\begin{aligned}
(N_{F/K},\lambda_{F/K})(\pmb{\phi}_{L/F}^{(\varphi_F)}(\sigma)) &=
\left(\pi_K^m.N_{L_0^{(K)}/K}{L_0^{(K)}}^\times,\phi_{L/L_0^{(K)}}
^{(\varphi_K^d)}(\varphi_F^{-m}\sigma)\right) \\
& =\left(\pi_K^{m'}.N_{L_0^{(K)}/K}{L_0^{(K)}}^\times,\phi_{L/L_0^{(K)}}
^{(\varphi_K^d)}(\varphi_K^{-m'}\sigma)\right) \\
&=\pmb{\phi}_{L/K}^{(\varphi_K)}(\sigma),
\end{aligned}
\end{equation*} 
which completes the proof.
\end{proof}
Let $L/K$ be any $APF$-Galois sub-extension of $K_{\varphi^d}/K$, where
the residue-class degree is $d$. In case $L/K$ is assumed
to be a finite extension, the $\widetilde{K}$-coordinate of the generalized
arrow $\pmb{\phi}_{L/K}^{(\varphi)} : \text{Gal}(L/K)\rightarrow
K^\times/N_{L_0/K}L_0^\times\times U_{\widetilde{\mathbb X}(L/K)}^\diamond
/U_{\mathbb X(L/K)}$ is the Iwasawa-Neukirch map $\iota_{L/K}$ of $L/K$
(for details on Iwasawa-Neukirch map $\iota_{L/K}$ of the Galois extension
$L/K$, cf. Section 1 of \cite{ikeda-serbest}). 
More precisely, we have the following proposition.
\begin{proposition}
Define a homomorphism
\begin{equation}
\rho : K^\times/N_{L_0/K}L_0^\times\times U_{\widetilde{\mathbb X}(L/K)}
^\diamond/U_{\mathbb X(L/K)}\rightarrow K^\times/N_{L/K}L^\times
\end{equation}
by
\begin{equation}
\rho : (\pi_K^m,(u_{\widetilde{E}}))\mapsto
\pi_K^m N_{L_0/K}(u_{\widetilde{L}_0})\mod{N_{L/K}L^\times},
\end{equation}
for every $(\pi_K^m,(u_{\widetilde{E}}))\in  K^\times/N_{L_0/K}L_0^\times
\times U_{\widetilde{\mathbb X}(L/K)}^\diamond/U_{\mathbb X(L/K)}$. Then
the composite map
\begin{equation}
\SelectTips{cm}{}\xymatrix{
{\text{Gal}(L/K)}\ar[r]^-{{\pmb\phi}_{L/K}^{(\varphi)}}
\ar@/^3pc/[rr]^{\rho\circ {\pmb\phi}_{L/K}^{(\varphi)}=\iota_{L/K}} 
& {K^\times/N_{L_0/K}L_0^\times\times 
U_{\widetilde{\mathbb X}{(L/K)}}^\diamond/U_{\mathbb X(L/K)}}
\ar[r]^-{\rho} & 
K^\times/N_{L/K}L^\times 
}
\end{equation}
is the Iwasawa-Neukirch map $\iota_{L/K} : \text{Gal}(L/K)\rightarrow 
K^\times/N_{L/K}L^\times$ of $L/K$.
\end{proposition}
\begin{proof}
Let us briefly recall, following Section 1 of \cite{ikeda-serbest}, 
the construction of the Iwasawa-Neukirch map
\begin{equation*}
\iota_{L/K}:\text{Gal}(L/K)\rightarrow K^\times/N_{L/K}L^\times
\end{equation*} 
for the Galois extension $L/K$. For each $\sigma\in\text{Gal}(L/K)$, choose 
$\sigma^*\in\text{Gal}(L^{nr}/K)$ in such a way that:
\begin{itemize}
\item[(i)]
$\sigma^*\mid_L=\sigma$;
\item[(ii)]
$\sigma^*\mid_{K^{nr}}=\varphi^n$, for some $0<n\in\mathbb Z$. 
\end{itemize}
Let $\Sigma_{\sigma^*}$ be the fixed-field $(L^{nr})^{\sigma^*}$ of
$\sigma^*\in\text{Gal}(L^{nr}/K)$ in $L^{nr}$. It is well-known that
$[\Sigma_{\sigma^*}:K]<\infty$. Now, the map 
$\iota_{L/K}:\text{Gal}(L/K)\rightarrow K^\times/N_{L/K}L^\times$ is defined
by $\iota_{L/K}(\sigma)=N_{\Sigma_{\sigma^*}/K}(\pi_{\Sigma_{\sigma^*}})
\mod{N_{L/K}L^\times}$, for $\sigma\in\text{Gal}(L/K)$, where 
$\pi_{\Sigma_{\sigma^*}}$ denotes any prime element of $\Sigma_{\sigma^*}$. 
Thus, for a finite $APF$-Galois extension $L/K$ satisfying
$[\kappa_L :\kappa_K]=d$ and $K\subset L\subset K_{\varphi^d}$, it suffices 
to prove that, for $\sigma\in\text{Gal}(L/K)$,
\begin{equation*}
\rho\circ\pmb{\phi}_{L/K}^{(\varphi)}(\sigma)=\iota_{L/K}(\sigma)=
N_{\Sigma_{\sigma^*}/K}(\pi_{\Sigma_{\sigma^*}})\mod{N_{L/K}L^\times},
\end{equation*}
where $\pi_{\Sigma_{\sigma^*}}$ denotes any prime element of 
$\Sigma_{\sigma^*}$.
Now, for $\sigma\in\text{Gal}(L/K)$, there exists $0\leq m\in\mathbb Z$ such
that $\sigma\mid_{L_0}=\varphi^m$ and $\tau=\varphi^{-m}\sigma\in
\text{Gal}(L/L_0)$. Thus, $\sigma=\varphi^m\tau$.
\begin{itemize}
\item[\textit{Case 1:}] $m>0$.
In this case, it suffices to prove that
\begin{equation*}
\pi_K^m N_{L_0/K}\left(\text{Pr}_{\widetilde{L}_0}
(\phi_{L/L_0}^{(\varphi^d)}(\varphi^{-m}\sigma))\right)=
N_{\Sigma_{\sigma^*}/K}(\pi_{\Sigma_{\sigma^*}})\mod{N_{L/K}L^\times},
\end{equation*}
where $\pi_{\Sigma_{\sigma^*}}$ denotes any prime element of 
$\Sigma_{\sigma^*}$.
To prove this equality, choose $\sigma^*\in\text{Gal}(L^{nr}/K)$ such that
\begin{itemize}
\item[(i)]
$\sigma^*\mid_L=\sigma$;
\item[(ii)]
$\sigma^*\mid_{K^{nr}}=\varphi^m$.
\end{itemize} 
In fact, let $\sigma^*=\varphi^m\mid_{L^{nr}}\tau^*$, where 
$\tau^*\in\text{Gal}(L^{nr}/L_0)$ is defined uniquely by the conditions
$\tau^*\mid_{L}=\tau$ and $\tau^*\mid_{K^{nr}}=\text{id}_{K^{nr}}$, as
$L^{nr}=LK^{nr}$. 
Note that, for $\Sigma=\Sigma_{\sigma^*}$, $\Sigma_0=\Sigma\cap K^{nr}$ 
is a finite extension of degree $[\Sigma_0:K]=m$ over $K$, as $\Sigma_0=
(K^{nr})^{\varphi^m}$, the fixed field of $\varphi^m\in\text{Gal}(K^{nr}/K)$ 
in $K^{nr}$.
As $L$ is fixed by $\varphi^d$, $T=L\cap K_\varphi$ is an unramified 
extension of degree $d$. Thus, the prime element $\pi_T$ is a prime element
of $L$ and of $L^{nr}$. Now, choose a prime element $\pi_\Sigma$ of $\Sigma$.
It is well-known that $\Sigma^{nr}=L^{nr}$ (cf. Section 2 in Chapter 4 of
\cite{fesenko-vostokov}). Thus, $\pi_\Sigma$ is a prime
element of $L^{nr}$. So, there exists a unit $v\in L^{nr}\subset\widetilde{L}$,
such that $\pi_\Sigma=\pi_T v$. Note that, $\pi_{\Sigma}^{\sigma^*-1}=1$
as $\Sigma$ is fixed by $\sigma^*$. Thus, $(\pi_T v)^{\sigma^*-1}=1$ and
we get the equalities
\begin{equation*}
\pi_T^{\sigma-1}=v^{1-\sigma^*}=v^{1-\varphi^m\tau^*}
=v^{1-\tau^*}v^{(1-\varphi^m)\tau^*}.
\end{equation*}
Recall that (by Proposition 1.8 of Chapter IV of \cite{fesenko-vostokov} or
by 1.1 of \cite{koch-deshalit}), $U_{\widetilde{L}}$ is multiplicatively 
$(1-\varphi^m)$-divisible. So, there exists $w\in U_{\widetilde{L}}$ such
that $w^{1-\varphi^m}=v$. Hence,
\begin{equation*}
\pi_T^{\sigma-1}=(w^{1-\tau^*}v^{\tau^*})^{1-\varphi^m},
\end{equation*}
as $\varphi^m\tau^*=\tau^*\varphi^m$. Now, choose $z\in U_{\widetilde{L}}$ as
$z=(w^{1-\tau^*}v^{\tau^*})^{1+\varphi+\cdots +\varphi^{m-1}}$. Note that,
this $z\in U_{\widetilde{L}}$ satisfies $z^{1-\varphi}=\pi_T^{\sigma-1}$. 
Clearly,
\begin{equation*}
\widetilde{N}_{L/K}(z)=\widetilde{N}_{L/K}(v)^{1+\varphi+\cdots+\varphi^{m-1}}.
\end{equation*}
After this preliminary observations,
\begin{equation*}
\begin{aligned}
N_{\Sigma/K}(\pi_\Sigma) &= N_{\Sigma_0/K}\circ 
N_{\Sigma/\Sigma_0}(\pi_\Sigma)\\
&=\widetilde{N}_{L/K}(\pi_\Sigma)^{1+\varphi+\cdots+\varphi^{m-1}}\\
&=\widetilde{N}_{L/K}(\pi_T v)^{1+\sigma+\cdots+\sigma^{m-1}}\\
&=\pi_K^m\widetilde{N}_{L/K}(v)^{1+\varphi+\cdots+\varphi^{m-1}}\\
&=\pi_K^m\widetilde{N}_{L/K}(z),
\end{aligned}
\end{equation*}
as $\pi_T$ belongs to the fixed Lubin-Tate labelling. Thus, the image
of $\sigma$ under the Iwasawa-Neukirch map $\iota_{L/K}$ is 
\begin{equation*}
\iota_{L/K}(\sigma)=\pi_K^m\widetilde{N}_{L/K}(z)\mod{N_{L/K}L^\times}.
\end{equation*}
Now, let $y\in U_L$ such that
\begin{equation*}
y^{1-\varphi^d}=\pi_T^{\varphi^{-m}\sigma-1}=\pi_T^{\sigma-1}.
\end{equation*}
Note that, $T=L\cap K_\varphi$. Then, setting $z=y^{1+\varphi+\cdots 
+\varphi^{d-1}}\in U_L$,
\begin{equation*}
z^{1-\varphi}=y^{1-\varphi^d}=\pi_T^{\sigma-1}.
\end{equation*}
Thus,
\begin{equation*}
N_{L/K}(y)=\widetilde{N}_{L/K}(y)^{1+\varphi+\cdots +\varphi^{d-1}}=
\widetilde{N}_{L/K}(z),
\end{equation*}
which shows that
\begin{equation*}
\begin{aligned}
\iota_{L/K}(\sigma)&=\pi_K^m\widetilde{N}_{L/K}(z)\mod{N_{L/K}L^\times}\\
&=\pi_K^m N_{L/K}(y)\mod{N_{L/K}L^\times}\\
&=\pi_K^m N_{L_0/K}\left(\widetilde{N}_{L/K}(y)\right)\mod{N_{L/K}L^\times}\\
&=\pi_K^m N_{L_0/K}\left(\text{Pr}_{\widetilde{L}_0}
(\phi_{L/L_0}^{(\varphi^d)}(\varphi^{-m}\sigma))\right)\\
&=\rho\circ\pmb{\phi}_{L/K}^{(\varphi)}(\sigma),
\end{aligned}
\end{equation*}
completing the proof.
\item[\textit{Case 2:}] $m=0$. In this case, $\sigma\in\text{Gal}(L/L_0)$.
Consider $\varphi^d\sigma\in\text{Gal}(L^{nr}/K)$. Then by the previous Case 1,
\begin{equation*}
\iota_{L/K}(\varphi^d\sigma)=\rho\circ\pmb{\phi}_{L/K}^{(\varphi)}
(\varphi^d\sigma),
\end{equation*}
where $\iota_{L/K}(\varphi^d\sigma)=\iota_{L/K}(\sigma)$. Now, by 
Theorem \ref{cocycle1}, 
\begin{equation*}
\pmb{\phi}_{L/K}^{(\varphi)}(\varphi^d\sigma)=
\pmb{\phi}_{L/K}^{(\varphi)}(\varphi^d)\pmb{\phi}_{L/K}^{(\varphi)}(\sigma)
^{\varphi^d}=\left(\pi_K^d.N_{L_0/K}L_0^\times,\phi_{L/L_0}^{(\varphi^d)}
(\sigma)\right),
\end{equation*}
where the last equality follows from the fact that 
$K\subset L\subset K_{\varphi^d}$. Thus,
\begin{equation*}
\rho\circ\pmb{\phi}_{L/K}^{(\varphi)}(\sigma)=
\rho\circ\pmb{\phi}_{L/K}^{(\varphi)}(\varphi^d\sigma)=\pi_K^dN_{L_0/K}
\left(\text{Pr}_{\widetilde{L}_0}(\phi_{L/L_0}^{(\varphi^d)}(\sigma))\right)
\mod{N_{L/K}L^\times},
\end{equation*}
which proves that
\begin{equation*}
\iota_{L/K}(\sigma)=\rho\circ\pmb{\phi}_{L/K}^{(\varphi)}(\sigma),
\end{equation*}
completing the proof.
\end{itemize}
\end{proof}
Now, we shall generalize the definition of the extended Hazewinkel map
$H_{L/K}^{(\varphi)}: U_{\widetilde{\mathbb X}{(L/K)}}^\diamond/Y_{L/K}
\rightarrow\text{Gal}(L/K)$ of Fesenko (cf. \cite{fesenko2000, fesenko2001,
fesenko2005} and \cite{ikeda-serbest}) initially defined for totally-ramified
$APF$-Galois sub-extensions $L/K$ of $K_\varphi/K$ to infinite $APF$-Galois
sub-extensions $L/K$ of $K_{\varphi^d}/K$, where $[\kappa_L:\kappa_K]=d$.

In order to do so, we first have to assume that the local field $K$
satisfies the condition
\begin{equation}
\label{rootofunity}
\pmb{\mu}_p(K^{sep})=\{\alpha\in K^{sep}:\alpha^p=1\}\subset K,
\end{equation}
where $p=\text{char}(\kappa_K)$. 
For details on the assumption (\ref{rootofunity}) on $K$, we refer the reader
to \cite{fesenko2000, fesenko2001, fesenko2005}.   

Let $L/K$ be an infinite $APF$-Galois extension with residue-class degree
$[\kappa_L:\kappa_K]=d$ and $K\subset L\subset K_{\varphi^d}$. As usual, let
$L_0=L\cap K^{nr}$. Define the generalized arrow
\begin{equation}
\label{generalized-hazewinkel-map}
\pmb{H}_{L/K}^{(\varphi)}:
K^\times/N_{L_0/K}L_0^\times\times U_{\widetilde{\mathbb X}(L/K)}^\diamond/
Y_{L/L_0}\rightarrow\text{Gal}(L/K)
\end{equation}
for the extension $L/K$ by
\begin{equation}
\label{generalized-hazewinkel-definition}
\pmb{H}_{L/K}^{(\varphi)}\left((\pi_K^m N_{L_0/K}L_0^\times, 
U.Y_{L/L_0})\right)=\varphi^m\mid_L H_{L/L_0}^{(\varphi^d)}(U.Y_{L/L_0}),
\end{equation}
for every $m\in\mathbb Z$ and $U\in U_{\widetilde{\mathbb X}(L/K)}^\diamond$,
where $H_{L/L_0}^{(\varphi^d)}:U_{\widetilde{\mathbb X}(L/K)}^\diamond/
Y_{L/L_0}\rightarrow\text{Gal}(L/L_0)$ is the extended Hazewinkel map of 
Fesenko for the extension $L/L_0$. For the definition and basic properties
of the group $Y_{L/L_0}$, we refer the reader to \cite{fesenko2005} and 
\cite{ikeda-serbest}.

The following lemma is clear.
\begin{lemma}
\label{bijection}
Suppose that the local field $K$ satisfies the condition given in eq.
(\ref{rootofunity}). Let $L/K$ be an infinite $APF$-Galois sub-extension
of $K_{\varphi^d}/K$, where $d=[\kappa_L:\kappa_K]$. Then the generalized 
arrow
\begin{equation}
\pmb{H}_{L/K}^{(\varphi)}:
K^\times/N_{L_0/K}L_0^\times\times U_{\widetilde{\mathbb X}(L/K)}^\diamond/
Y_{L/L_0}\rightarrow\text{Gal}(L/K)
\end{equation}
for the extension $L/K$ is a bijection.
\end{lemma}
\begin{proof}
The proof follows from the isomorphism
$\text{Gal}(L/K)\simeq\text{Gal}(L_0/K)\times\text{Gal}(L/L_0)$ combined with
the bijectivity of $H_{L/L_0}^{(\varphi^d)}:U_{\widetilde{\mathbb X}(L/K)}
^\diamond/Y_{L/L_0}\rightarrow\text{Gal}(L/L_0)$ (cf. Lemma 5.22 of 
\cite{ikeda-serbest}) and abelian local class field theory for the extension 
$L_0/K$.
\end{proof}
Now, consider the composition of arrows
\begin{equation}
\SelectTips{cm}{}\xymatrix{
{\text{Gal}(L/K)}\ar[r]^-{\pmb{\phi}_{L/K}^{(\varphi)}}
\ar@{.>}[dr]_-{\pmb{\Phi}_{L/K}^{(\varphi)}} & {K^\times/N_{L_0/K}L_0^\times
\times U_{\widetilde{\mathbb X}(L/K)}^\diamond/U_{\mathbb X(L/K)}}
\ar[d]^{(\text{id}_{K^\times/N_{L_0/K}L_0^\times}, c_{L/L_0})} \\
{} & {K^\times/N_{L_0/K}L_0^\times
\times U_{\widetilde{\mathbb X}(L/K)}^\diamond/Y_{L/L_0}}
}
\end{equation}
where $c_{L/L_0}:U_{\widetilde{\mathbb X}(L/K)}^\diamond/U_{\mathbb X(L/K)}
\rightarrow U_{\widetilde{\mathbb X}(L/K)}^\diamond/Y_{L/L_0}$ is the
canonical map defined via the inclusion 
$U_{\mathbb X(L/K)}\subseteq Y_{L/L_0}$. 
Recall that (cf. eq. 5.35 of \cite{ikeda-serbest}), the composition
$c_{L/L_0}\circ\phi_{L/L_0}^{(\varphi^d)}=\Phi_{L/L_0}^{(\varphi^d)}:
\text{Gal}(L/L_0)\rightarrow U_{\widetilde{\mathbb X}(L/K)}^\diamond
/Y_{L/L_0}$ is the reciprocity map of Fesenko for the extension $L/L_0$.
Now, let $\sigma\in\text{Gal}(L/K)$.
Let $0\leq m\in\mathbb Z$ such that $\sigma\mid_{L_0}=\varphi^m\mid_{L_0}$
and $\varphi^{-m}\sigma\in\text{Gal}(L/L_0)$. Then following the definition,
\begin{equation*}
\begin{aligned}
\pmb{H}_{L/K}^{(\varphi)}\circ\pmb{\Phi}_{L/K}^{(\varphi)}(\sigma) &=
\pmb{H}_{L/K}^{(\varphi)}\left(\pi_K^m.N_{L_0/K}L_0^\times,
c_{L/L_0}\circ\phi_{L/L_0}^{(\varphi^d)}(\varphi^{-m}\sigma)\right)\\
&=\pmb{H}_{L/K}^{(\varphi)}\left(\pi_K^m.N_{L_0/K}L_0^\times,
\Phi_{L/L_0}^{(\varphi^d)}(\varphi^{-m}\sigma)\right)\\
&=\varphi^m\mid_L H_{L/L_0}^{(\varphi^d)}\left(\Phi_{L/L_0}^{(\varphi^d)}
(\varphi^{-m}\sigma)\right)\\
&=\varphi^m\mid_L (\varphi^{-m}\sigma)\\
&=\sigma,
\end{aligned}
\end{equation*}
by Lemma 5.23 of \cite{ikeda-serbest}.
For $0\leq m\in\mathbb Z$ and $U\in U_{\widetilde{\mathbb X}(L/K)}
^\diamond$, let $(\pi_K^m.N_{L_0/K}L_0^\times,U.Y_{L/L_0})\in
K^\times/N_{L_0/K}L_0^\times\times U_{\widetilde{\mathbb X}(L/K)}
^\diamond/Y_{L/L_0}$. Now, again following the definition,
\begin{equation*}
\begin{aligned}
\pmb{\Phi}_{L/K}^{(\varphi)}\circ\pmb{H}_{L/K}^{(\varphi)}\left(
(\pi_K^m.N_{L_0/K}L_0^\times,U.Y_{L/L_0})\right)&=
\pmb{\Phi}_{L/K}^{(\varphi)}\left(\varphi^m\mid_{L}H_{L/L_0}^{(\varphi^d)}
(U.Y_{L/L_0})\right) \\
&=(\text{id}_{K^\times/N_{L_0/K}L_0^\times},c_{L/L_0})\circ
\pmb{\phi}_{L/K}^{(\varphi)}
\left(\varphi^m\mid_{L}H_{L/L_0}^{(\varphi^d)}(U.Y_{L/L_0})\right)\\
&=(\text{id}_{K^\times/N_{L_0/K}L_0^\times},c_{L/L_0})
\left(\pi_K^m.N_{L_0/K}L_0^\times,\phi_{L/L_0}^{(\varphi^d)}(H_{L/L_0}
^{(\varphi^d)}(U.Y_{L/L_0})\right)\\
&=\left(\pi_K^m.N_{L_0/K}L_0^\times,\Phi_{L/L_0}^{(\varphi^d)}(H_{L/L_0}
^{(\varphi^d)}(U.Y_{L/L_0})\right)\\
&=\left(\pi_K^m.N_{L_0/K}L_0^\times,U.Y_{L/L_0}\right),
\end{aligned}
\end{equation*}
by Lemma 5.23 of \cite{ikeda-serbest}. Thus, these computations yield
\begin{equation}
\label{identity1}
\pmb{H}_{L/K}^{(\varphi)}\circ\pmb{\Phi}_{L/K}^{(\varphi)}
=\text{id}_{\text{Gal}(L/K)}
\end{equation}
and
\begin{equation}
\label{identity2}
\pmb{\Phi}_{L/K}^{(\varphi)}\circ\pmb{H}_{L/K}^{(\varphi)}
=\text{id}_{K^\times/N_{L_0/K}L_0^\times\times U_{\widetilde{\mathbb X}(L/K)}
^\diamond/Y_{L/L_0}}.
\end{equation}

Note that, there is a natural continuous action of $\text{Gal}(L/K)$ on the
topological group  
$K^\times/N_{L_0/K}L_0^\times\times U_{\widetilde{\mathbb X}(L/K)}^\diamond
/Y_{L/L_0}$ 
defined by abelian local class field theory on
the first component and by eq. (5.7) and Lemma 5.20 of \cite{ikeda-serbest}
on the second component as
\begin{equation}
\label{galoismodule2}
(\overline{a},\overline{U})^\sigma=
\left(\overline{a}^{\varphi^m},\overline{U}^{\varphi^{-m}\sigma}\right)=
\left(\overline{a},\overline{U}^{\varphi^{-m}\sigma}\right),
\end{equation}
for every $\sigma\in\text{Gal}(L/K)$, where $\sigma\mid_{L_0}=\varphi^m$
for some $0\leq m\in\mathbb Z$, and
for every $a\in K^\times$ with $\overline{a}=a.N_{L_0/K}L_0^\times$ 
and $U\in U_{\widetilde{\mathbb X}(L/K)}^\diamond$ with
$\overline{U}=U.Y_{L/L_0}$.
\textit{We shall always view 
$K^\times/N_{L_0/K}L_0^\times
\times U_{\widetilde{\mathbb X}(L/K)}^\diamond/Y_{L/L_0}$
as a topological $\text{Gal}(L/K)$-module in this text}.

So, we have the following theorem, which follows from Theorem \ref{cocycle1},
Lemma \ref{bijection} and eq.s (\ref{identity1}) and (\ref{identity2})
combined with the fact that $U_{\mathbb X(L/K)}$ is a topological
$\text{Gal}(L/L_0)$-submodule of $Y_{L/L_0}$.
\begin{theorem}
\label{cocycle2}
Suppose that the local field $K$ satisfies the condition given in eq.
(\ref{rootofunity}). Let $L/K$ be an infinite $APF$-Galois sub-extension
of $K_{\varphi^d}/K$, where $d=[\kappa_L:\kappa_K]$. The mapping
\begin{equation*}
\pmb{\Phi}_{L/K}^{(\varphi)}:\text{Gal}(L/K)\rightarrow 
K^\times/N_{L_0/K}L_0^\times\times U_{\widetilde{\mathbb X}{(L/K)}}
^\diamond/Y_{L/L_0}
\end{equation*}
defined for the extension $L/K$ is a bijection with the inverse 
\begin{equation*}
\pmb{H}_{L/K}^{(\varphi)}:K^\times/N_{L_0/K}L_0^\times\times 
U_{\widetilde{\mathbb X}{(L/K)}}^\diamond/Y_{L/L_0}\rightarrow
\text{Gal}(L/K).
\end{equation*} 
For every $\sigma,\tau\in\text{Gal}(L/K)$,
\begin{equation}
\pmb{\Phi}_{L/K}^{(\varphi)}(\sigma\tau)=\pmb{\Phi}_{L/K}^{(\varphi)}(\sigma)
\pmb{\Phi}_{L/K}^{(\varphi)}(\tau)^\sigma
\end{equation}
co-cycle condition is satisfied. 
\end{theorem}
By Corollary \ref{grouplaw1}, Theorem \ref{cocycle2} has the following
consequence.
\begin{corollary}
\label{grouplaw2}
Define a law of composition $\ast$ on 
$K^\times/N_{L_0/K}L_0^\times\times U_{\widetilde{\mathbb X}{(L/K)}}
^\diamond/Y_{L/L_0}$ by
\begin{equation}
\label{star2}
(\overline{a},\overline{U})\ast(\overline{b},\overline{V})=
(\overline{a},\overline{U}).(\overline{b},\overline{V})
^{(\pmb{\Phi}_{L/K}^{(\varphi)})^{-1}((\overline{a},\overline{U}))}
\end{equation}
for every $\overline{a}=a.N_{L_0/K}L_0^\times, \overline{b}
=b.N_{L_0/K}L_0^\times\in K^\times/N_{L_0/K}L_0^\times$ with $a,b\in K^\times$
and $\overline{U}=U.Y_{L/L_0}, \overline{V}=V.Y_{L/L_0}\in 
U_{\widetilde{\mathbb X}{(L/K)}}^\diamond/Y_{L/L_0}$ with $U,V\in
U_{\widetilde{\mathbb X}{(L/K)}}^\diamond$. Then
$K^\times/N_{L_0/K}L_0^\times\times U_{\widetilde{\mathbb X}{(L/K)}}
^\diamond/Y_{L/L_0}$ is a topological 
group under $\ast$, and the map $\pmb{\Phi}_{L/K}^{(\varphi)}$ induces 
an isomorphism of topological groups
\begin{equation}
\pmb{\Phi}_{L/K}^{(\varphi)}:\text{Gal}(L/K)\xrightarrow{\sim}
K^\times/N_{L_0/K}L_0^\times\times U_{\widetilde{\mathbb X}{(L/K)}}
^\diamond/Y_{L/L_0}, 
\end{equation}
where the topological group structure on 
$K^\times/N_{L_0/K}L_0^\times\times U_{\widetilde{\mathbb X}{(L/K)}}
^\diamond/Y_{L/L_0}$ is defined with
respect to the binary operation $\ast$ defined by eq. (\ref{star2}).
\end{corollary}
\begin{definition}
Let $K$ be a local field satisfying the condition given in eq. 
(\ref{rootofunity}).
Let $L/K$ be an infinite $APF$-Galois sub-extension of $K_{\varphi^d}/K$,
where $d=[\kappa_L:\kappa_K]$.
The mapping 
\begin{equation*}
\pmb{\Phi}_{L/K}^{(\varphi)}:\text{Gal}(L/K)\rightarrow
K^\times/N_{L_0/K}L_0^\times\times U_{\widetilde{\mathbb X}{(L/K)}}
^\diamond/Y_{L/L_0},
\end{equation*} 
defined in Theorem \ref{cocycle2}, is called the 
\textit{generalized Fesenko reciprocity map for the extension $L/K$}.
\end{definition}
For each $0\leq i\in\mathbb R$, we have previously introduced the higher
unit groups
$\left(U_{\widetilde{\mathbb X}(L/K)}^\diamond\right)^i$ of the field
$\widetilde{\mathbb X}(L/K)$ as in eq. (\ref{diamond-group-i}). 
For each $0\leq n\in\mathbb Z$, as in eq. 5.42 of \cite{ikeda-serbest}, let
\begin{equation}
\label{Q-group}
Q_{L/L_0}^n=c_{L/L_0}\left(\left(U_{\widetilde{\mathbb X}(L/K)}
^\diamond\right)^n
U_{\mathbb X(L/K)}/U_{\mathbb X(L/K)}\cap\text{im}(\phi_{L/L_0}
^{(\varphi^d)})\right),
\end{equation}
which is a subgroup of $\left(U_{\widetilde{\mathbb X}(L/K)}^\diamond\right)
^nY_{L/L_0}/Y_{L/L_0}$. Now, ramification theorem, stated in Theorem 
\ref{ramification1}, can be reformulated for the generalized
reciprocity map $\pmb{\Phi}_{L/K}^{(\varphi)}$ corresponding to the extension
$L/K$ as follows.
\begin{theorem}[Ramification theorem]
\label{ramification2} 
Let $K$ be a local field satisfying the condition given in eq.
(\ref{rootofunity}).
For $0\leq u\in\mathbb R$,
let $\text{Gal}(L/K)_u$ denote the $u^{th}$ higher ramification 
subgroup in the lower numbering of the Galois group $\text{Gal}(L/K)$ 
corresponding to the infinite $APF$-Galois sub-extension $L/K$ of 
$K_{\varphi^d}/K$ with residue-class degree $[\kappa_L:\kappa_K]=d$. 
Then, for $0\leq n\in\mathbb Z$, there exists the inclusion
\begin{multline*}
\pmb{\Phi}_{L/K}^{(\varphi)}
\left(\text{Gal}(L/K)_{\psi_{L/K}\circ\varphi_{L/L_0}(n)}
-\text{Gal}(L/K)_{\psi_{L/K}\circ\varphi_{L/L_0}(n+1)}\right)
\subseteq \\
\left<1_{K^\times/N_{L_0/K}L_0^\times}\right>\times
\left(
\left(U_{\widetilde{\mathbb X}{(L/K)}}^\diamond\right)^{n} 
Y_{L/L_0}/Y_{L/L_0}-Q_{L/L_0}^{n+1}\right).
\end{multline*}
\end{theorem}
\begin{proof}
Following the general observation made in the first paragarph of the
proof of Theorem \ref{ramification1}, for $0\leq u\in\mathbb R$ and 
for $\tau\in\text{Gal}(L/K)_u=\text{Gal}(L/L_0)^{\varphi_{L/K}(u)}$,
\begin{equation*}
\pmb{\phi}_{L/K}^{(\varphi)}(\tau) = 
\left(1_{K^\times/N_{L_0/K}L_0^\times}, 
\phi_{L/L_0}^{(\varphi')}(\tau)\right),
\end{equation*}
where $\varphi'=\varphi^d$. Thus, following the definition,
\begin{equation*}
\begin{aligned}
\pmb{\Phi}_{L/K}^{(\varphi)}(\tau) &= 
\left(1_{K^\times/N_{L_0/K}L_0^\times}, 
c_{L/L_0}\circ\phi_{L/L_0}^{(\varphi')}(\tau)\right)\\
&=\left(1_{K^\times/N_{L_0/K}L_0^\times}, 
\Phi_{L/L_0}^{(\varphi')}(\tau)\right).
\end{aligned}
\end{equation*}
Now, to prove the theorem, let
$u=\psi_{L/K}\circ\varphi_{L/L_0}(n)$
and $u'=\psi_{L/K}\circ\varphi_{L/L_0}(n+1)$. Then, for
any $\tau\in\text{Gal}(L/K)_u-\text{Gal}(L/K)_{u'}$,
it follows from the ramification theorem 
(cf. Theorem 5.27 in \cite{ikeda-serbest}) that
the second coordinate of $\pmb{\Phi}_{L/K}^{(\varphi)}(\tau)$
satisfies 
\begin{equation*}
\Phi_{L/L_0}^{(\varphi')}(\tau)\in
\left(U_{\widetilde{\mathbb X}{(L/K)}}^\diamond\right)^{n}Y_{L/L_0}/Y_{L/L_0}-
Q_{L/L_0}^{n+1},
\end{equation*}
since 
\begin{equation*}
\text{Gal}(L/K)_u=\text{Gal}(L/L_0)^{\varphi_{L/K}(u)}
=\text{Gal}(L/L_0)^{\varphi_{L/L_0}(n)}=\text{Gal}(L/L_0)_n
\end{equation*} 
and likewise
\begin{equation*}
\text{Gal}(L/K)_{u'}=\text{Gal}(L/L_0)^{\varphi_{L/K}(u')}
=\text{Gal}(L/L_0)^{\varphi_{L/L_0}(n+1)}=\text{Gal}(L/L_0)_{n+1},
\end{equation*}
which completes the proof.
\end{proof}
Let $K$ be a local field satisfying the condition given in 
eq. (\ref{rootofunity}).
Let $M/K$ be an infinite Galois sub-extension of $L/K$. Thus, by 
Lemma 3.3 of \cite{ikeda-serbest}, $M$ is an $APF$-Galois extension over $K$.
We further assume that, the residue-class degree $[\kappa_M:\kappa_K]=d'$ and 
$K\subset M\subset K_{\varphi^{d'}}$ for some $d'\mid d$. 
Let
\begin{equation*}
\pmb{\Phi}_{M/K}^{(\varphi)}:\text{Gal}(M/K)\rightarrow
K^\times/N_{M_0/K}M_0^\times
\times U_{\widetilde{\mathbb X}(M/K)}^\diamond/Y_{M/M_0}
\end{equation*}
be the corresponding generalized Fesenko reciprocity map defined for the 
extension $M/K$. Here, $M_0$ is defined as usual by 
$M_0=M\cap K^{nr}=K_{d'}^{nr}$. 
Now, fix a basic sequence 
\begin{equation*}
L_o=E_o\subset E_1\subset\cdots\subset E_i\subset\cdots\subset L
\end{equation*}
for the extension $L/L_0$. Now, following the notation of \cite{fesenko2005}
and \cite{ikeda-serbest}, introduce for each $1\leq i\in\mathbb Z$,
an element $\sigma_i$ in $\text{Gal}(\widetilde{L}/\widetilde{K})$ such
that $\left<\sigma\mid_{E_i}\right>=\text{Gal}(E_i/E_{i-1})$. Further,
for each $1\leq k, i\in\mathbb Z$,
introduce the map $h_k^{(L/L_o)}:\prod_{1\leq i\leq k}U^{\sigma_i-1}
_{\widetilde{E}_k}\rightarrow\left(\prod_{1\leq i\leq k+1}
U^{\sigma_i-1}_{\widetilde{E}_{k+1}}\right)/U^{\sigma_{k+1}-1}
_{\widetilde{E}_{k+1}}$, the map
$g_k^{(L/L_o)}:\prod_{1\leq i\leq k}U_{\widetilde{E}_k}^{\sigma_i-1}\rightarrow
\prod_{1\leq i\leq k+1}U_{\widetilde{E}_{k+1}}^{\sigma_i-1}$ and the map
$f_i^{(L/L_o)}:U_{\widetilde{E}_i}^{\sigma_i-1}\rightarrow 
U_{\widetilde{\mathbb X}(L/E_i)}\xrightarrow{\Lambda_{E_i/E_o}} 
U_{\widetilde{\mathbb X}(L/K)}$
following \cite{fesenko2005} and \cite{ikeda-serbest}.
Now, fix the sequence
\begin{equation*}
M_o=E_o\cap M\subseteq E_1\cap M\subseteq\cdots\subseteq E_i\cap M\subseteq
\cdots\subseteq M
\end{equation*}
for the extension $M/M_o$. Define, for each $1\leq k\in\mathbb Z$,
a homomorphism
\begin{equation*}
h_k^{(M/M_o)}:\prod_{1\leq i\leq k}U^{\sigma_i\mid_{\widetilde{M}}-1}
_{\widetilde{E_k\cap M}}\rightarrow\left(\prod_{1\leq i\leq k+1}
U^{\sigma_i\mid_{\widetilde M}-1}_{\widetilde{E_{k+1}\cap M}}\right)
/U^{\sigma_{k+1}\mid_{\widetilde M}-1}_{\widetilde{E_{k+1}\cap M}}
\end{equation*}
that satisfies
\begin{equation*}
\left(\prod_{0\leq\ell\leq f(L/M)-1}(\varphi^{d'})^\ell
\widetilde{N}_{E_{k+1}/E_{k+1}\cap M}\right)\circ 
h_k^{(L/L_o)}=h_k^{(M/M_o)}\circ\left(\prod_{0\leq\ell\leq f(L/M)-1}
(\varphi^{d'})^\ell\widetilde{N}_{E_k/E_k\cap M}\right)
\end{equation*}
and any map
\begin{equation*}
g_k^{(M/M_o)}:
\prod_{1\leq i\leq k}U^{\sigma_i\mid_{\widetilde{M}}-1}
_{\widetilde{E_k\cap M}}\rightarrow\prod_{1\leq i\leq k+1}
U^{\sigma_i\mid_{\widetilde M}-1}_{\widetilde{E_{k+1}\cap M}}
\end{equation*}
that satisfies
\begin{equation*}
\left(\prod_{0\leq\ell\leq f(L/M)-1}(\varphi^{d'})^\ell
\widetilde{N}_{E_{k+1}/E_{k+1}\cap M}\right)\circ 
g_k^{(L/L_o)}=g_k^{(M/M_o)}\circ\left(\prod_{0\leq\ell\leq f(L/M)-1}
(\varphi^{d'})^\ell\widetilde{N}_{E_k/E_k\cap M}\right)
\end{equation*}
following the same lines of \cite{fesenko2005} and \cite{ikeda-serbest}. 

Now, for each $1\leq i\in\mathbb Z$, introduce the map
\begin{equation*}
f_i^{(M/M_o)}: U_{\widetilde{E_i\cap M}}^{\sigma_i\mid_{\widetilde{M}}-1}
\rightarrow U_{\widetilde{\mathbb X}(M/K)}
\end{equation*}
by
\begin{equation*}
f_i^{(M/M_o)}(w)=\widetilde{\mathcal N}_{L/M}\left(f_i^{(L/L_o)}(v)\right),
\end{equation*}
where $v\in U_{\widetilde{E}_i}^{\sigma_i-1}$ is any element satisfying
$\prod_{0\leq\ell\leq f(L/M)-1}(\varphi^{d'})^\ell
\widetilde{N}_{E_i/E_i\cap M}(v)=w
\in U_{\widetilde{E_i\cap M}}^{\sigma_i\mid_{\widetilde{M}}-1}$. 
Note that, if 
$v'\in U_{\widetilde{E}_i}^{\sigma_i-1}$ such that
$\prod_{0\leq\ell\leq f(L/M)-1}(\varphi^{d'})^\ell
\widetilde{N}_{E_i/E_i\cap M}(v')=w$, then
$\widetilde{\mathcal N}_{L/M}\left(f_i^{(L/L_o)}(v)\right)=
\widetilde{\mathcal N}_{L/M}\left(f_i^{(L/L_o)}(v')\right)$. 

In fact, there 
exists $u\in\ker\left(\prod_{0\leq\ell\leq f(L/M)-1}(\varphi^{d'})
^\ell\widetilde{N}_{E_i/E_i\cap M}\right)$ such that
$v'=vu$. Thus, we have to verify that
$\widetilde{\mathcal N}_{L/M}\left(f_i^{(L/L_o)}(v)\right)=
\widetilde{\mathcal N}_{L/M}\left(f_i^{(L/L_o)}(vu)\right)$. That is,
for each $1\leq j\in\mathbb Z$, we have to check that
\begin{multline*}
\prod_{0\leq\ell\leq f(L/M)-1}(\varphi^{d'})^\ell
\widetilde{N}_{E_j/E_j\cap M}\left(\text{Pr}_{\widetilde{E}_j}
(f_i^{(L/L_o)}(v))\right)=\\
\prod_{0\leq\ell\leq f(L/M)-1}(\varphi^{d'})^\ell
\widetilde{N}_{E_j/E_j\cap M}\left(\text{Pr}_{\widetilde{E}_j}
(f_i^{(L/L_o)}(vu))\right).
\end{multline*}
Now, for $j>i$, it follows that
\begin{equation*}
\begin{aligned}
\prod_{0\leq\ell\leq f(L/M)-1}(\varphi^{d'})^\ell
\widetilde{N}_{E_j/E_j\cap M}\left(\text{Pr}_{\widetilde{E}_j}
(f_i^{(L/L_o)}(v))\right)=\\
\prod_{0\leq\ell\leq f(L/M)-1}(\varphi^{d'})^\ell
\widetilde{N}_{E_j/E_j\cap M}\left(g_{j-1}^{(L/L_o)}\circ\cdots
\circ g_i^{(L/L_o)}(v)\right)=\\
g_{j-1}^{(M/M_o)}\circ\cdots\circ g_i^{(M/M_o)}
\left(\prod_{0\leq\ell\leq f(L/M)-1}(\varphi^{d'})^\ell\widetilde{N}
_{E_i/E_i\cap M}(v)\right)=\\
g_{j-1}^{(M/M_o)}\circ\cdots\circ g_i^{(M/M_o)}
\left(\prod_{0\leq\ell\leq f(L/M)-1}(\varphi^{d'})^\ell\widetilde{N}
_{E_i/E_i\cap M}(vu)\right)=\\
\prod_{0\leq\ell\leq f(L/M)-1}(\varphi^{d'})^\ell
\widetilde{N}_{E_j/E_j\cap M}\left(g_{j-1}^{(L/L_o)}\circ\cdots
\circ g_i^{(L/L_o)}(vu)\right)=\\
\prod_{0\leq\ell\leq f(L/M)-1}(\varphi^{d'})^\ell
\widetilde{N}_{E_j/E_j\cap M}\left(\text{Pr}_{\widetilde{E}_j}
(f_i^{(L/L_o)}(vu))\right) .
\end{aligned}
\end{equation*}
Thus, the map 
\begin{equation*}
f_i^{(M/M_o)}: U_{\widetilde{E_i\cap M}}^{\sigma_i\mid_{\widetilde{M}}-1}
\rightarrow U_{\widetilde{\mathbb X}(M/K)}
\end{equation*}
is well-defined. Moreover, for $j>i$,
\begin{equation*}
\text{Pr}_{\widetilde{E_j\cap M}}\circ f_i^{(M/M_o)}=\left(g_{j-1}^{(M/M_o)}
\circ\cdots\circ g_i^{(M/M_o)}\right)\mid_{U_{\widetilde{E_i\cap M}}
^{\sigma_i\mid_{\widetilde{M}}-1}}.
\end{equation*}
In fact, for $w\in U_{\widetilde{E_i\cap M}}^{\sigma_i\mid_{\widetilde{M}}-1}$,
there exists $v\in U_{\widetilde{E}_i}^{\sigma_i-1}$ such that
$\prod_{0\leq\ell\leq f(L/M)-1}(\varphi^{d'})^\ell\widetilde{N}
_{E_i/E_i\cap M}(v)=w$, and 
$f_i^{(M/M_o)}(w)=\widetilde{\mathcal N}_{L/M}\left(f_i^{(L/L_o)}(v)\right)$. 
That is, the following square
\begin{equation}
\label{square-L/M}
\SelectTips{cm}{}\xymatrix{
{U_{\widetilde{E}_i}^{\sigma_i-1}}\ar[r]^{f_i^{(L/L_o)}}
\ar[d]_{\prod_{0\leq\ell\leq f(L/M)-1}(\varphi^{d'})^\ell\widetilde{N}
_{E_i/E_i\cap M}} & 
{U_{\widetilde{\mathbb X}(L/K)}}\ar[d]^{\widetilde{\mathcal N}_{L/M}} \\
{U_{\widetilde{E_i\cap M}}^{\sigma_i\mid_{\widetilde{M}}-1}}
\ar[r]^{f_i^{(M/M_o)}} & {U_{\widetilde{\mathbb X}(M/K)}}
}
\end{equation}
is commutative. Thus,
\begin{equation*}
\begin{aligned}
\text{Pr}_{\widetilde{E_j\cap M}}\circ f_i^{(M/M_o)}(w)&=
\text{Pr}_{\widetilde{E_j\cap M}}\circ\widetilde{\mathcal N}_{L/M}
\left(f_i^{(L/L_o)}(v)\right)\\
&=\prod_{0\leq\ell\leq f(L/M)-1}(\varphi^{d'})^\ell\widetilde{N}
_{E_j/E_j\cap M}\left(\text{Pr}_{\widetilde{E}_j}
\circ f_i^{(L/L_o)}(v)\right)\\
&=\prod_{0\leq\ell\leq f(L/M)-1}(\varphi^{d'})^\ell\widetilde{N}
_{E_j/E_j\cap M}\left((g_{j-1}^{(L/L_o)}\circ\cdots
\circ g_i^{(L/L_o)})(v)\right)\\
&=\left(g_{j-1}^{(M/M_o)}\circ\cdots\circ g_i^{(M/M_o)}\right)
\left(\prod_{0\leq\ell\leq f(L/M)-1}(\varphi^{d'})^\ell\widetilde{N}
_{E_i/E_i\cap M}(v)\right),
\end{aligned}
\end{equation*}
which is the desired equality.

Now, we shall modify Lemma 5.28 of \cite{ikeda-serbest} and show that the
norm map $\widetilde{\mathcal N}_{L/M}:\widetilde{\mathbb X}(L/K)^\times
\rightarrow\widetilde{\mathbb X}(M/K)^\times$ introduced by eq.s 
(\ref{coleman-norm-map}) and (\ref{coleman-norm-map-definition}) satisfies
the following lemma.
\begin{lemma}
\label{Z-Y}
The norm map $\widetilde{\mathcal N}_{L/M}:
\widetilde{\mathbb X}(L/K)^\times\rightarrow\widetilde{\mathbb X}(M/K)^\times$ 
introduced by eq.s (\ref{coleman-norm-map}) and 
(\ref{coleman-norm-map-definition}) further satisfies
\begin{itemize}
\item[(i)]
$\widetilde{\mathcal N}_{L/M}\left(Z_{L/L_0}\left(\{E_i,f_i^{(L/L_o)}\}\right)
\right)\subseteq Z_{M/M_0}\left(\{E_i\cap M,f_i^{(M/M_o)}\}\right)$;
\item[(ii)]
$\widetilde{\mathcal N}_{L/M}\circ\left<\varphi\right>_{L/M}
\left(Y_{L/L_0}\right)\subseteq Y_{M/M_0}$.
\end{itemize}
\end{lemma}
\begin{proof}
\begin{itemize}
\item[(i)]
Recall that, $\widetilde{\mathcal N}_{L/M}:
\widetilde{\mathbb X}(L/K)^\times\rightarrow\widetilde{\mathbb X}(M/K)^\times$ 
is a continuous mapping. Now, for any choice of 
$z^{(i)}\in\text{im}(f_i^{(L/L_o)})$, the continuity of the multiplicative 
arrow $\widetilde{\mathcal N}_{L/M}:
\widetilde{\mathbb X}(L/K)^\times\rightarrow\widetilde{\mathbb X}(M/K)^\times$ 
yields
\begin{equation*}
\widetilde{\mathcal N}_{L/M}\left(\prod_iz^{(i)}\right)=\prod_i\widetilde
{\mathcal N}_{L/M}(z^{(i)}),
\end{equation*}
where $\widetilde{\mathcal N}_{L/M}(z^{(i)})\in\text{im}(f_i^{(M/M_o)})$
by the commutative square (\ref{square-L/M}).
\item[(ii)]
For $y\in Y_{L/L_0}$, as $y^{1-\varphi^d}\in Z_{L/L_0}$, it follows that 
$\widetilde{\mathcal N}_{L/M}(y^{1-\varphi^d})\in Z_{M/M_0}$ by part (i).
Now, note that
\begin{equation*}
\widetilde{\mathcal N}_{L/M}(y^{1-\varphi^d})=\widetilde{\mathcal N}_{L/M}(y)^
{1-\varphi^d}=\left(\widetilde{\mathcal N}_{L/M}(y)
^{1+\varphi^{d'}+\cdots\varphi^{d'(f(L/M)-1)}}\right)^{1-\varphi^{d'}}.
\end{equation*} 
Therefore,
\begin{equation*}
\widetilde{\mathcal N}_{L/M}(y)^{1+\varphi^{d'}+\cdots\varphi^{d'(f(L/M)-1)}}
=\widetilde{\mathcal N}_{L/M}\circ\left<\varphi\right>_{L/M}(y)\in Y_{M/M_0}
\end{equation*}
as desired.
\end{itemize}
The proof is now complete.
\end{proof}
Thus, by part (ii) of Lemma \ref{Z-Y}, the homomorphism
$\widetilde{\mathcal N}_{L/M}\circ\left<\varphi\right>_{L/M} :
\widetilde{\mathbb X}(L/K)^\times\rightarrow\widetilde{\mathbb X}(M/K)^\times$
induces a group homomorphism, which will again be called the \textit{Coleman
norm map from $L$ to $M$},
\begin{equation}
\label{coleman-norm-2}
\widetilde{\mathcal N}_{L/M}^{\text{Coleman}}:
U^\diamond_{\widetilde{\mathbb X}(L/K)}/Y_{L/L_0}\rightarrow
U^\diamond_{\widetilde{\mathbb X}(M/K)}/Y_{M/M_0}
\end{equation}
and defined by
\begin{equation}
\label{coleman-norm-2-def}
\widetilde{\mathcal N}_{L/M}^{\text{Coleman}}\left(\overline{U}\right)=
\widetilde{\mathcal N}_{L/M}\circ\left<\varphi\right>_{L/M}\left(U\right).
Y_{M/M_0},
\end{equation}
for every $U\in U^\diamond_{\widetilde{\mathbb X}(L/K)}$, where
$\overline{U}$ denotes, as usual, the coset $U.Y_{L/L_0}$ in
$U^\diamond_{\widetilde{\mathbb X}(L/K)}/Y_{L/L_0}$.

The following lemma is the finer version of Lemma \ref{square1-preliminary}.
\begin{lemma}
\label{square3-preliminary}
Let $K$ be a local field satisfying the condition given in eq. 
(\ref{rootofunity}).
For an infinite Galois sub-extension $M/K$ of $L/K$ such that the 
residue-class degree $[\kappa_M:\kappa_K]=d'$ and 
$K\subset M\subset K_{\varphi^{d'}}$ for some $d'\mid d$,
the square
\begin{equation}
\label{square-L/M-Y}
\SelectTips{cm}{}\xymatrix{
{\text{Gal}(L/L_0)}\ar[r]^-{\Phi_{L/L_0}^{(\varphi^d)}}
\ar[d]_{\text{res}_M} & 
{U_{\widetilde{\mathbb X}{(L/L_0)}}^\diamond/Y_{L/L_0}}
\ar[d]^{\widetilde{\mathcal N}_{L/M}^{\text{Coleman}}} \\
{\text{Gal}(M/M_0)}\ar[r]^-{\Phi_{M/M_0}^{(\varphi^{d'})}} & 
{U_{\widetilde{\mathbb X}{(M/M_0)}}^\diamond/Y_{M/M_0}},
}
\end{equation}
where the right-vertical arrow is the Coleman norm map 
$\widetilde{\mathcal N}_{L/M}^{\text{Coleman}}$ 
from $L$ to $M$ defined by eq.s 
(\ref{coleman-norm-2}) and (\ref{coleman-norm-2-def}), is commutative.
\end{lemma} 
\begin{proof}
It suffices to prove that the square
\begin{equation*}
\SelectTips{cm}{}\xymatrix{
{U_{\widetilde{\mathbb X}(L/K)}^\diamond/U_{\mathbb X(L/K)}}
\ar[r]^{c_{L/L_0}}\ar[d]_{\widetilde{\mathcal N}^{Coleman}_{L/M}} & 
{U_{\widetilde{\mathbb X}(L/K)}^\diamond/Y_{L/L_0}}
\ar[d]^{\widetilde{\mathcal N}^{Coleman}_{L/M}} \\
{U_{\widetilde{\mathbb X}(M/K)}^\diamond/U_{\mathbb X(M/K)}}
\ar[r]^{c_{M/M_0}} & 
{U_{\widetilde{\mathbb X}(M/K)}^\diamond/Y_{M/M_0}}
}
\end{equation*}
is commutative, which is obvious. Then pasting this square with the 
square eq. (\ref{square-L/M-U}) as
\begin{equation*}
\SelectTips{cm}{}\xymatrix{
{Gal(L/K)}\ar[r]^-{\phi_{L/L_0}^{(\varphi^d)}}\ar[d]_{\text{res}_M} & 
{U_{\widetilde{\mathbb X}(L/K)}^\diamond/U_{\mathbb X(L/K)}}
\ar[r]^{c_{L/L_0}}\ar[d]^{\widetilde{\mathcal N}^{Coleman}_{L/M}} & 
{U_{\widetilde{\mathbb X}(L/K)}^\diamond/Y_{L/K}}
\ar[d]^{\widetilde{\mathcal N}^{Coleman}_{L/M}} \\
{Gal(M/K)}\ar[r]^-{\phi_{M/M_0}^{(\varphi^{d'})}} & 
{U_{\widetilde{\mathbb X}(M/K)}^\diamond/U_{\mathbb X(M/K)}}
\ar[r]^{c_{M/M_0}} & 
{U_{\widetilde{\mathbb X}(M/K)}^\diamond/Y_{M/K}}
}
\end{equation*}
the commutativity of the square eq. (\ref{square-L/M-Y}) follows. 
\end{proof}
Thus, we have the following theorem, which is the finer version of Theorem
\ref{square1}.
\begin{theorem}
\label{square3}
Let $K$ be a local field satisfying the condition given in eq. 
(\ref{rootofunity}).
For an infinite Galois sub-extension $M/K$ of $L/K$ such that the 
residue-class degree $[\kappa_M:\kappa_K]=d'$ and 
$K\subset M\subset K_{\varphi^{d'}}$ for some $d'\mid d$,
the square
\begin{equation*}
\SelectTips{cm}{}\xymatrix{
{\text{Gal}(L/K)}\ar[r]^-{\pmb{\Phi}_{L/K}^{(\varphi)}}\ar[d]_{\text{res}_M} & 
{K^\times/N_{L_0/K}L_0^\times
\times
U_{\widetilde{\mathbb X}{(L/K)}}^\diamond
/Y_{L/L_0}}\ar[d]^{\left(e_{L_0/M_0}^{\text{CFT}},
\widetilde{\mathcal N}_{L/M}^{\text{Coleman}}\right)} \\
{\text{Gal}(M/K)}\ar[r]^-{\pmb{\Phi}_{M/K}^{(\varphi)}} & 
{K^\times/N_{M_0/K}M_0^\times
\times
U_{\widetilde{\mathbb X}{(M/K)}}^\diamond
/Y_{M/M_0}},
}
\end{equation*}
where the right-vertical arrow
\begin{equation*}
K^\times/N_{L_0/K}L_0^\times
\times U_{\widetilde{\mathbb X}{(L/K)}}^\diamond/Y_{L/L_0}
\xrightarrow{\left(e_{L_0/M_0}^{\text{CFT}},\widetilde{\mathcal N}_{L/M}
^{\text{Coleman}}\right)} 
K^\times/N_{M_0/K}M_0^\times
\times U_{\widetilde{\mathbb X}{(M/K)}}^\diamond/Y_{M/M_0}
\end{equation*}
defined by
\begin{equation*}
\left(e_{L_0/M_0}^{\text{CFT}},\widetilde{\mathcal N}_{L/M}
^{\text{Coleman}}\right) :
(\overline{a},\overline{U})\mapsto\left(e_{L_0/M_0}^{\text{CFT}}(\overline{a}),
\widetilde{\mathcal N}_{L/M}^{\text{Coleman}}(\overline{U})\right)
\end{equation*}
for every $(\overline{a},\overline{U})\in K^\times/N_{L_0/K}L_0^\times\times
U_{\widetilde{\mathbb X}(L/K)}^\diamond/Y_{L/L_0}$, is commutative.
Here, 
\begin{equation*}
e_{L_0/M_0}^{\text{CFT}}:K^\times/N_{L_0/K}L_0^\times\rightarrow
K^\times/N_{M_0/K}M_0^\times
\end{equation*} 
is the natural inclusion defined via the existence theorem of local 
class field theory.
\end{theorem}
\begin{proof}
By the isomorphism defined by eq.s (\ref{directproduct}) and 
(\ref{directproduct-definition}), for $\sigma\in\text{Gal}(L/K)$, there
exists a unique $0\leq m\in\mathbb Z$ such that $\sigma\mid_{L_0}=\varphi^m$
and $\varphi^{-m}\sigma\in\text{Gal}(L/L_0)$. Now, following the definition,
\begin{equation*}
\pmb{\Phi}_{L/K}^{(\varphi)}(\sigma) = \left(\pi_K^m N_{L_0/K}L_0^\times ,
\Phi_{L/L_0}^{(\varphi^d)}(\varphi^{-m}\sigma)\right).
\end{equation*}
Thus,
\begin{equation*}
\begin{aligned}
\left(e_{L_0/M_0}^{\text{CFT}},
\widetilde{\mathcal N}_{L/M}^{\text{Coleman}}\right)
\left(\pi_K^m N_{L_0/K}L_0^\times , 
\Phi_{L/L_0}^{(\varphi^d)}(\varphi^{-m}\sigma)\right)  =\\ 
\left(e_{L_0/M_0}^{\text{CFT}}(\pi_K^m N_{L_0/K}L_0^\times),
\widetilde{\mathcal N}_{L/M}^{\text{Coleman}}(\Phi_{L/L_0}
^{(\varphi^d)}(\varphi^{-m}\sigma))\right)  =\\
\left(\pi_K^m N_{M_0/K}M_0^\times,\Phi_{M/M_0}^{(\varphi^{d'})}
(\varphi^{-m}\sigma\mid_{M})\right) 
\end{aligned}
\end{equation*}
by Lemma \ref{square3-preliminary}.
Note that, by the existence theorem of local class field theory,
\begin{equation*}
e_{L_0/M_0}^{\text{CFT}}(\pi_K^m N_{L_0/K}L_0^\times)
=\pi_K^mN_{M_0/K}M_0^\times=\pi_K^{m'}N_{M_0/K}M_0^\times ,
\end{equation*}
where $0\leq m'\in\mathbb Z$ is the unique integer satisfying
$(\sigma\mid_M)\mid_{M_0}=\sigma\mid_{M_0}=\varphi^{m'}$ and 
$\varphi^{-m'}(\sigma\mid_{M})\in\text{Gal}(M/M_0)$.
Hence,
\begin{equation*}
\begin{aligned}
\left(e_{L_0/M_0}^{\text{CFT}},
\widetilde{\mathcal N}_{L/M}^{\text{Coleman}}\right)
(\pmb{\Phi}_{L/K}^{(\varphi)}(\sigma)) &= 
\left(\pi_K^{m'}N_{M_0/K}M_0^\times ,\Phi_{M/M_0}^{(\varphi^{d'})}
(\varphi^{-m}\sigma\mid_{M})\right)\\
& = \left(\pi_K^{m'}N_{M_0/K}M_0^\times ,\Phi_{M/M_0}^{(\varphi^{d'})}
(\varphi^{-m'}(\sigma\mid_{M}))\right)\\
&=\pmb{\Phi}_{M/K}^{(\varphi)}(\text{res}_M(\sigma))
\end{aligned}
\end{equation*}
by Remark \ref{directproduct-diagram} part (i), which completes the proof.
\end{proof}
Let $K$ be a local field satisfying the condition given in eq. 
(\ref{rootofunity}).
Now, let $F/K$ be a finite sub-extension of $L/K$. Thus, $L/F$ is an infinite
$APF$-Galois extension (cf. Lemma 3.3 of \cite{ikeda-serbest}), where
$F$ satisfies (\ref{rootofunity}). Fix a 
Lubin-Tate splitting $\varphi_F$ over $F$. Now, assume that the residue-class
degree $[\kappa_L : \kappa_F]=d'$, for some $d'\mid d$, 
and there exists the chain of field extensions 
\begin{equation*}
F\subset L\subset F_{(\varphi_F)^{d'}}.
\end{equation*}
Thus, there exists the generalized Fesenko reciprocity map
\begin{equation*}
\pmb{\Phi}_{L/F}^{(\varphi_F)}:\text{Gal}(L/F)\rightarrow
F^\times/N_{L_0^{(F)}/F}{L_0^{(F)}}^\times
\times
U_{\widetilde{\mathbb X}{(L/F)}}^\diamond/Y_{L/L_0^{(F)}}
\end{equation*}
corresponding to the extension $L/F$. Here, $L_0^{(F)}$ is defined as
usual by $L_0^{(F)}=L\cap F^{nr}=F^{nr}_{d'}$ (and recall that, $L_0^{(K)}=
L\cap K^{nr}=K^{nr}_d$).

Now, fix a basic sequence 
\begin{equation*}
L_0^{(K)}=E_o\subset E_1\subset\cdots\subset E_i\subset\cdots\subset L
\end{equation*}
for the extension $L/L_0^{(K)}$. Following the notation of \cite{fesenko2005}
and \cite{ikeda-serbest}, introduce for each $1\leq i\in\mathbb Z$,
an element $\sigma_i$ in $\text{Gal}(\widetilde{L}/\widetilde{K})$ such
that $\left<\sigma\mid_{E_i}\right>=\text{Gal}(E_i/E_{i-1})$. 
Now, fix the sequence
\begin{equation*}
L_0^{(F)}=E_oL_0^{(F)}\subseteq E_1L_0^{(F)}\subseteq\cdots\subseteq 
E_iL_0^{(F)}\subseteq\cdots\subseteq LL_0^{(F)}=L
\end{equation*}
for the extension $L/L_0^{(F)}$ following eq. (5.55) of \cite{ikeda-serbest}. 
Now, for $1\leq i\in\mathbb Z$, introduce the elements $\sigma_i^*$
in $\text{Gal}(\widetilde{L}/\widetilde{F})$ that satisfies
\begin{equation*} 
<\sigma_i^*\mid_{E_iL_0^{(F)}}>=\text{Gal}(E_iL_0^{(F)}/E_{i-1}L_0^{(F)})
\end{equation*}
as follows :
\begin{itemize}
\item[(i)] in case $i>i_o$, then define $\sigma_i^*=\sigma_i$ ; 
\item[(ii)] in case $i\leq i_o$, then define 
\begin{equation*}
\sigma_i^*=
\begin{cases}
\sigma_i, & E_{i-1}L_0^{(F)}\subset E_{i}L_0^{(F)} ; \\
\text{id}_{E_iL_0^{(F)}}, & E_{i-1}L_0^{(F)}=E_{i}L_0^{(F)} ;
\end{cases}
\end{equation*}
\end{itemize}
where $i_o$ is defined as in the paragarph of eq. (5.55) of 
\cite{ikeda-serbest}. It is then clear that, for each $1\leq i\in\mathbb Z$, 
the elements 
$\sigma_i^*$ of $\text{Gal}(\widetilde{L}/\widetilde{F})$
satisfies
\begin{equation*}
<\sigma_i^*\mid_{E_iL_0^{(F)}}>=\text{Gal}(E_iL_0^{(F)}/E_{i-1}L_o^{(F)}),
\end{equation*}
and for almost all $i$, $\sigma_i^*=\sigma_i$.
Further, for each $1\leq k, i\in\mathbb Z$,
introduce the map $h_k^{(L/L_0^{(F)})}:\prod_{1\leq i\leq k}U^{\sigma^*_i-1}
_{\widetilde{E_kL_0^{(F)}}}\rightarrow\left(\prod_{1\leq i\leq k+1}
U^{\sigma_i^*-1}_{\widetilde{E_{k+1}L_0^{(F)}}}\right)/U^{\sigma^*_{k+1}-1}
_{\widetilde{E_{k+1}L_0^{(F)}}}$, the map
$g_k^{(L/L_0^{(F)})}:\prod_{1\leq i\leq k}U_{\widetilde{E_kL_0^{(F)}}}
^{\sigma^*_i-1}\rightarrow\prod_{1\leq i\leq k+1}U_{\widetilde{E_{k+1}
L_0^{(F)}}}^{\sigma_i^*-1}$ and the map
$f_i^{(L/L_0^{(F)})}:U_{\widetilde{E_iL_0^{(F)}}}^{\sigma^*_i-1}\rightarrow 
U_{\widetilde{\mathbb X}(L/E_iL_0^{(F)})}
\xrightarrow{\Lambda_{E_iL_0^{(F)}/L_0^{(F)}}} 
U_{\widetilde{\mathbb X}(L/F)}$
following \cite{fesenko2005} and \cite{ikeda-serbest}.
Define, for each $1\leq k\in\mathbb Z$, a homomorphism
\begin{equation*}
h_k^{(L/L_0^{(K)})}:\prod_{1\leq i\leq k}U^{\sigma_i-1}
_{\widetilde{E}_k}\rightarrow\left(\prod_{1\leq i\leq k+1}
U^{\sigma_i-1}_{\widetilde{E}_{k+1}}\right)
/U^{\sigma_{k+1}-1}_{\widetilde{E}_{k+1}}
\end{equation*}
that satisfies
\begin{equation*}
\widetilde{N}^*_{E_{k+1}L_0^{(F)}/E_{k+1}}\circ 
h_k^{(L/L_0^{(F)})}=h_k^{(L/L_0^{(K)})}\circ\widetilde{N}_{E_kL_0^{(F)}/E_k}
\end{equation*}
and any map
\begin{equation*}
g_k^{(L/L_0^{(K)})}:
\prod_{1\leq i\leq k}U^{\sigma_i-1}
_{\widetilde{E}_k}\rightarrow\prod_{1\leq i\leq k+1}
U^{\sigma_i-1}_{\widetilde{E}_{k+1}}
\end{equation*}
that satisfies
\begin{equation*}
\widetilde{N}_{E_{k+1}L_0^{(F)}/E_{k+1}}\circ g_k^{(L/L_0^{(F)})}=
g_k^{(L/L_0^{(K)})}\circ\widetilde{N}_{E_kL_0^{(F)}/E_k}
\end{equation*}
following the same lines of \cite{fesenko2005} and \cite{ikeda-serbest}. 

Now, for each $1\leq i\in\mathbb Z$, introduce the map
\begin{equation*}
f_i^{(L/L_0^{(K)})}: U_{\widetilde{E}_i}^{\sigma_i-1}
\rightarrow U_{\widetilde{\mathbb X}(L/K)}
\end{equation*}
by
\begin{equation*}
f_i^{(L/L_0^{(K)})}(w)
=\Lambda_{F/K}\left(f_i^{(L/L_0^{(F)})}(v)\right),
\end{equation*}
where $v\in U_{\widetilde{E_iL_0^{(F)}}}^{\sigma^*_i-1}$ is any element 
satisfying
$\widetilde{N}_{E_iL_0^{(F)}/E_i}(v)=w\in U_{\widetilde{E}_i}^{\sigma_i-1}$. 
Note that, if 
$v'\in U_{\widetilde{E_iL_0^{(F)}}}^{\sigma^*_i-1}$ such that
$\widetilde{N}_{E_iL_0^{(F)}/E_i}(v')=w$, then
$\Lambda_{F/K}\left(f_i^{(L/L_0^{(F)})}(v)\right)=
\Lambda_{F/K}\left(f_i^{(L/L_0^{(F)})}(v')\right)$. 

In fact, there exists $u\in\ker\left(\widetilde{N}_{E_iL_0^{(F)}/E_i}\right)$ 
such that $v'=vu$. Thus, we have to verify that
$\Lambda_{F/K}\left(f_i^{(L/L_0^{(F)})}(v)\right)=
\Lambda_{F/K}\left(f_i^{(L/L_0^{(F)})}(vu)\right)$. That is,
for each $1\leq j\in\mathbb Z$, we have to check that
\begin{equation}
\label{projection-lambda-equality}
\text{Pr}_{\widetilde{E}_j}\left(\Lambda_{F/K}(f_i^{(L/L_0^{(F)})}(v))\right)=
\text{Pr}_{\widetilde{E}_j}\left(\Lambda_{F/K}(f_i^{(L/L_0^{(F)})}(vu))\right).
\end{equation}
Now, for $j>i$, it follows that
\begin{equation*}
\begin{aligned}
\text{Pr}_{\widetilde{E}_j}\left(\Lambda_{F/K}(f_i^{(L/L_0^{(F)})}(v))\right)
&=
\widetilde{N}_{E_jL_0^{(F)}/E_j}\left(\text{Pr}_{\widetilde{E_jL_0^{(F)}}}
\left(\Lambda_{F/K}(f_i^{(L/L_0^{(F)})}(v))\right)\right)\\
&=\widetilde{N}_{E_jL_0^{(F)}/E_j}\left(\text{Pr}_{\widetilde{E_jL_0^{(F)}}}
\left(f_i^{(L/L_0^{(F)})}(v)\right)\right)\\
&=\widetilde{N}_{E_jL_0^{(F)}/E_j}\left(g_{j-1}^{(L/L_0^{(F)})}\circ\cdots
\circ g_i^{(L/L_0^{(F)})}(v)\right)\\
&=g_{j-1}^{(L/L_0^{(K)})}\circ\cdots\circ g_i^{(L/L_0^{(K)})}
\left(\widetilde{N}_{E_iL_0^{(F)}/E_i}(v)\right),
\end{aligned}
\end{equation*}
by the properties of the mappings $g_k^{(L/L_0^{(F)})}$ and 
$g_k^{(L/L_0^{(K)})}$.
Thus, eq. (\ref{projection-lambda-equality}) follows, as  
$\widetilde{N}_{E_iL_0^{(F)}/E_i}(v)=\widetilde{N}_{E_iL_0^{(F)}/E_i}(vu)$.
Therefore, the map
\begin{equation*}
f_i^{(L/L_0^{(K)})}: U_{\widetilde{E}_i}^{\sigma_i-1}
\rightarrow U_{\widetilde{\mathbb X}(L/K)}
\end{equation*}
is well-defined. Moreover, for $j>i$,
\begin{equation*}
\text{Pr}_{\widetilde{E}_j}\circ f_i^{(L/L_0^{(K)})}=
\left(g_{j-1}^{(L/L_0^{(K)})}\circ\cdots\circ g_i^{(L/L_0^{(K)})}\right)
\mid_{U_{\widetilde{E}_i}^{\sigma_i-1}}.
\end{equation*}
In fact, for $w\in U_{\widetilde{E}_i}^{\sigma_i-1}$,
there exists $v\in U_{\widetilde{E_iL_0^{(F)}}}^{\sigma^*_i-1}$ such that
$\widetilde{N}_{E_iL_0^{(F)}/E_i}(v)=w$, and 
$f_i^{(L/L_0^{(K)})}(w)=\Lambda_{F/K}\left(f_i^{(L/L_0^{(F)})}(v)\right)$. 
That is, the following square
\begin{equation}
\label{square-F/K}
\SelectTips{cm}{}\xymatrix{
{U_{\widetilde{E_iL_0^{(F)}}}^{\sigma^*_i-1}}\ar[r]^{f_i^{(L/L_0^{(F)})}}
\ar[d]_{\widetilde{N}_{E_iL_0^{(F)}/E_i}} & 
{U_{\widetilde{\mathbb X}(L/F)}}\ar[d]^{\Lambda_{F/K}} \\
{U_{\widetilde{E}_i}^{\sigma_i-1}}
\ar[r]^{f_i^{(L/L_0^{(K)})}} & {U_{\widetilde{\mathbb X}(L/K)}}
}
\end{equation}
is commutative. Thus,
\begin{equation*}
\begin{aligned}
\text{Pr}_{\widetilde{E}_j}\circ f_i^{(L/L_0^{(K)})}(w)&=
\text{Pr}_{\widetilde{E}_j}\circ\Lambda_{F/K}
\left(f_i^{(L/L_0^{(F)})}(v)\right)\\
&=\widetilde{N}_{E_jL_0^{(F)}/E_j}\left(\text{Pr}_{\widetilde{E_jL_0^{(F)}}}
\circ f_i^{(L/L_0^{(F)})}(v)\right)\\
&=\widetilde{N}_{E_jL_0^{(F)}/E_j}\left((g_{j-1}^{(L/L_0^{(F)})}\circ\cdots
\circ g_i^{(L/L_0^{(F)})})(v)\right)\\
&=\left(g_{j-1}^{(L/L_0^{(K)})}\circ\cdots\circ g_i^{(L/L_0^{(K)})}\right)
\left(\widetilde{N}_{E_iL_0^{(F)}/E_i}(v)\right),
\end{aligned}
\end{equation*}
by the properties of the mappings $g_k^{(L/L_0^{(F)})}$ and 
$g_k^{(L/L_0^{(K)})}$, which is the desired equality.

Now, we shall modify Lemma 5.30 of \cite{ikeda-serbest} and show that
the homomorphism $\Lambda_{F/K}:\widetilde{\mathbb X}(L/L_0^{(F)})^\times
\rightarrow\widetilde{\mathbb X}(L/L_0^{(K)})^\times$ introduced by 
eq.s (\ref{Lambda-map}) and (\ref{Lambda-map-definition}) satisfies the 
following lemma.
\begin{lemma}
The continuous homomorphism
$\Lambda_{F/K}:\widetilde{\mathbb X}(L/L_0^{(F)})^\times\rightarrow
\widetilde{\mathbb X}(L/L_0^{(K)})^\times$ introduced by eq.s 
(\ref{Lambda-map}) and (\ref{Lambda-map-definition}) further satisfies
\begin{itemize}
\item[(i)] 
$\Lambda_{F/K}\left(Z_{L/L_0^{(F)}}(\{K_iF,f_i^{(L/L_0^{(F)})}\})\right)
\subseteq Z_{L/L_0^{(K)}}\left(\{K_i,f_i^{(L/L_0^{(K)})}\}\right)$;
\item[(ii)]
$\Lambda_{F/K}(Y_{L/L_0^{(F)}})\subseteq Y_{L/L_0^{(K)}}$.
\end{itemize}
\end{lemma}
\begin{proof} 
\begin{itemize}
\item[(i)]
For any choice of $z^{(i)}\in Z_i^{(L/L_0^{(F)})}$, the continuity of the 
multiplicative arrow 
$\Lambda_{F/K}:\widetilde{\mathbb X}(L/L_0^{(F)})^\times\rightarrow
\widetilde{\mathbb X}(L/L_0^{(K)})^\times$ yields
\begin{equation*}
\Lambda_{F/K}\left(\prod_iz^{(i)}\right)=\prod_i\Lambda_{F/K}(z^{(i)}),
\end{equation*}
where $\Lambda_{F/K}(z^{(i)})\in Z_i^{(L/L_0^{(K)})}$ by the 
commutative square (\ref{square-F/K}).
\item[(ii)]
Now let $y\in Y_{L/L_0^{(F)}}$. Then, 
$y^{1-\varphi_F^{d'}}\in Z_{L/L_0^{(F)}}(\{K_iF,f_i^{(L/L_0^{(F)})}\})$. 
Thus,
$\Lambda_{F/K}(y^{1-\varphi_F^{d'}})
=\Lambda_{F/K}(y)^{1-\varphi_F^{d'}}\in Z_{L/L_0^{(K)}}
\left(\{K_i,f_i^{(L/L_0^{(K)})}\}\right)$ by part (i). 
Now the result follows, as $\varphi_F^{d'}=\varphi_K^d$ by Remark 
\ref{varphi-varphi}.
\end{itemize}
\end{proof}
Thus, the homomorphism
$\Lambda_{F/K}:\widetilde{\mathbb X}(L/L_0^{(F)})^\times\rightarrow
\widetilde{\mathbb X}(L/L_0^{(K)})^\times$ of eq. (\ref{Lambda-map}) 
defined by eq. (\ref{Lambda-map-definition}) induces a group homomorphism,
\begin{equation}
\label{lambda-map-2}
\lambda_{F/K}:
U_{\widetilde{\mathbb X}(L/L_0^{(F)})}^\diamond/Y_{L/L_0^{(F)}}
\rightarrow U_{\widetilde{\mathbb X}(L/L_0^{(K)})}^\diamond/Y_{L/L_0^{(K)}}
\end{equation}
and defined by
\begin{equation}
\label{lambda-map-definition-2}
\lambda_{F/K}(\overline{U})=
\Lambda_{F/K}(U).Y_{L/L_0^{(K)}},
\end{equation}
for every $U\in U_{\widetilde{\mathbb X}(L/L_0^{(F)})}^\diamond$, where
$\overline{U}$ denotes, as usual, the coset $U.Y_{L/L_0^{(F)}}$ in
$U_{\widetilde{\mathbb X}(L/L_0^{(F)})}^\diamond/Y_{L/L_0^{(F)}}$.

Let
\begin{equation*}
\pmb{\Phi}_{L/F}^{(\varphi_F)}:\text{Gal}(L/F)\rightarrow
F^\times/N_{L_0^{(F)}/F}{L_0^{(F)}}^\times\times 
U_{\widetilde{\mathbb X}{(L/L_0^{(F)})}}^\diamond/Y_{L/L_0^{(F)}}
\end{equation*}
be the corresponding generalized Fesenko reciprocity map defined for the 
extension $L/F$, where $Y_{L/L_0^{(F)}}=Y_{L/L_0^{(F)}}
\left(\{K_iF,f_i^{(L/L_0^{(F)})}\}\right)$. 

The following lemma is the finer version of Lemma \ref{square2-preliminary}.
\begin{lemma}
\label{square4-preliminary}
Let $K$ be a local field satisfying the condition given in eq. 
(\ref{rootofunity}).
Let $F/K$ be a finite sub-extension of $L/K$. Fix a Lubin-Tate splitting
$\varphi_F$ over $F$. Assume that the residue-class degree 
$[\kappa_L:\kappa_F]=d'$ and $F\subset L\subset F_{(\varphi_F)^{d'}}$ for
some $d'\mid d$. Then the square
\begin{equation}
\label{square-F/K-Y-preliminary}
\SelectTips{cm}{}\xymatrix{
{\text{Gal}(L/L_0^{(F)})}\ar[r]^-{\Phi_{L/L_0^{(F)}}^{(\varphi_K^d)}}
\ar[d]_{\text{inc.}} & 
{U_{\widetilde{\mathbb X}{(L/L_0^{(F)})}}^\diamond/Y_{L/L_0^{(F)}}}
\ar[d]^{\lambda_{F/K}} \\
{\text{Gal}(L/L_0^{(K)})}\ar[r]^-{\Phi_{L/L_0^{(K)}}^{(\varphi_K^d)}} & 
{U_{\widetilde{\mathbb X}{(L/L_0^{(K)})}}^\diamond/Y_{L/L_0^{(K)}}},
}
\end{equation}
where the right-vertical arrow
\begin{equation*}
\lambda_{F/K}:U_{\widetilde{\mathbb X}{(L/L_0^{(F)})}}^\diamond/Y_{L/L_0^{(F)}}
\rightarrow U_{\widetilde{\mathbb X}{(L/L_0^{(K)})}}^\diamond/Y_{L/L_0^{(K)}}
\end{equation*}
is defined by eq.s (\ref{lambda-map-2}) and (\ref{lambda-map-definition-2}), 
is commutative.
\end{lemma}
\begin{proof}
It suffices to prove that the square
\begin{equation*}
\SelectTips{cm}{}\xymatrix{
{U_{\widetilde{\mathbb X}(L/L_0^{(F)})}^\diamond/U_{\mathbb X(L/L_0^{(F)})}}
\ar[r]^{\text{can.}}\ar[d]_{\lambda_{F/K}} & 
{U_{\widetilde{\mathbb X}(L/L_0^{(F)})}^\diamond/Y_{L/L_0^{(F)}}}
\ar[d]^{\lambda_{F/K}} \\
{U_{\widetilde{\mathbb X}(L/L_0^{(K)})}^\diamond/U_{\mathbb X(L/L_0^{(K)})}}
\ar[r]^{\text{can.}} & 
{U_{\widetilde{\mathbb X}(L/L_0^{(K)})}^\diamond/Y_{L/L_0^{(K)}}}
}
\end{equation*}
is commutative, which is obvious. Then pasting this square with the 
square eq. (\ref{square-F/K-U-preliminary}) as
\begin{equation*}
\SelectTips{cm}{}\xymatrix{
{Gal(L/L_0^{(F)})}\ar[r]^-{\Phi_{L/L_0^{(F)}}^{(\varphi_K^d)}}
\ar[d]_{\text{inc.}} & 
{U_{\widetilde{\mathbb X}(L/L_0^{(F)})}^\diamond/U_{\mathbb X(L/L_0^{(F)})}}
\ar[r]^{\text{can.}}\ar[d]^{\lambda_{L/F}} & 
{U_{\widetilde{\mathbb X}(L/L_0^{(F)})}^\diamond/Y_{L/L_0^{(F)}}}
\ar[d]^{\lambda_{L/F}} \\
{Gal(L/L_0^{(K)})}\ar[r]^-{\Phi_{L/L_0^{(K)}}^{(\varphi_K^d)}} & 
{U_{\widetilde{\mathbb X}(L/L_0^{(K)})}^\diamond/U_{\mathbb X(L/L_0^{(K)})}}
\ar[r]^{\text{can.}} & 
{U_{\widetilde{\mathbb X}(L/L_0^{(K)})}^\diamond/Y_{L/L_0^{(K)}}}
}
\end{equation*}
the commutativity of the square eq. (\ref{square-F/K-Y-preliminary}) follows. 
\end{proof}
Thus, we have the following theorem, which is the finer version 
of Theorem \ref{square2}.
\begin{theorem}
\label{square4}
Let $K$ be a local field satisfying the condition given in eq. 
(\ref{rootofunity}).
Let $F/K$ be a finite sub-extension of $L/K$. Fix a 
Lubin-Tate splitting $\varphi_F$ over $F$. Assume that the residue-class
degree $[\kappa_L : \kappa_F]=d'$ 
and $F\subset L\subset F_{(\varphi_F)^{d'}}$ for some $d'\mid d$. 
Then the square
\begin{equation}
\label{square-F/K-Y}
\SelectTips{cm}{}\xymatrix{
{\text{Gal}(L/F)}\ar[r]^-{\pmb{\Phi}_{L/F}^{(\varphi_F)}}
\ar[d]_{\text{inc.}} & 
{F^\times/N_{L_0^{(F)}/F}{L_0^{(F)}}^\times
\times
U_{\widetilde{\mathbb X}{(L/F)}}^\diamond
/Y_{L/L_0^{(F)}}}\ar[d]^{(N_{F/K},\lambda_{F/K})} \\
{\text{Gal}(L/K)}\ar[r]^-{\pmb{\Phi}_{L/K}^{(\varphi_K)}} & 
{K^\times/N_{L_0^{(K)}/K}{L_0^{(K)}}^\times
\times
U_{\widetilde{\mathbb X}{(L/K)}}^\diamond
/Y_{L/L_0^{(K)}}},
}
\end{equation}
where the right-vertical arrow
\begin{multline*}
(N_{F/K},\lambda_{F/K}) :
F^\times/N_{L_0^{(F)}/F}{L_0^{(F)}}^\times
\times
U_{\widetilde{\mathbb X}(L/F)}^\diamond/Y_{L/L_0^{(F)}}
\rightarrow \\
K^\times/N_{L_0^{(K)}/K}{L_0^{(K)}}^\times
\times
U_{\widetilde{\mathbb X}{(L/K)}}^\diamond
/Y_{L/L_0^{(K)}}
\end{multline*}
defined by
\begin{equation*}
(N_{F/K},\lambda_{F/K}) :
(\overline{a},\overline{U})\mapsto\left(\overline{N_{F/K}(a)},
\lambda_{F/K}(\overline{U})\right),
\end{equation*}
for every $(\overline{a},\overline{U})\in 
{F^\times/N_{L_0^{(F)}/F}{L_0^{(F)}}^\times}\times
U_{\widetilde{\mathbb X}(L/F)}^\diamond/Y_{L/L_0^{(F)}}$, 
is commutative.
\end{theorem}
\begin{proof}
Let $\sigma\in\text{Gal}(L/F)$. There exists $0\leq m\in\mathbb Z$ such that
$\sigma\mid_{L_0^{(F)}}=\varphi_F^m$ and $\varphi_F^{-m}\sigma\in
\text{Gal}(L/L_0^{(F)})$. Now,
\begin{equation*}
\pmb{\Phi}_{L/F}^{(\varphi_F)}(\sigma)=
\left(\pi_F^m.N_{L_0^{(F)}/F}{L_0^{(F)}}^\times,\Phi_{L/L_0^{(F)}}
^{(\varphi_K^d)}(\varphi_F^{-m}\sigma)\right)
\end{equation*}
and
\begin{equation*}
(N_{F/K},\lambda_{F/K})(\pmb{\Phi}_{L/F}^{(\varphi_F)}(\sigma))=
\left(\pi_K^m.N_{L_0^{(K)}/K}{L_0^{(K)}}^\times,\Phi_{L/L_0^{(K)}}
^{(\varphi_K^d)}(\varphi_F^{-m}\sigma)\right)
\end{equation*}
by the norm-compatibility of primes in the fixed Lubin-Tate labelling and by
Lemma \ref{square4-preliminary}. Now, there exists $0\leq m'\in\mathbb Z$ 
such that $\sigma\mid_{L_0^{(K)}}=\varphi_K^{m'}$ and 
$\varphi_K^{-m'}\sigma\in\text{Gal}(L/L_0^{(K)})$. 
By Remark \ref{directproduct-diagram} part (ii), it follows that
$\varphi_F^m\mid_{L_0^{(K)}}=\varphi_K^{m'}$ and $\varphi_F^{-m}\sigma=
\varphi_K^{-m'}\sigma$. By abelian local class field theory,
$N_{F/K} : \pi_F^m N_{L_0^{(F)}/F}{L_0^{(F)}}^\times\mapsto
\pi_K^{m'}N_{L_0^{(K)}/K}{L_0^{(K)}}^\times
=\pi_K^m.N_{L_0^{(K)}/K}{L_0^{(K)}}^\times $. Thus,
\begin{equation*}
\begin{aligned}
(N_{F/K},\lambda_{F/K})(\pmb{\Phi}_{L/F}^{(\varphi_F)}(\sigma)) &=
\left(\pi_K^m.N_{L_0^{(K)}/K}{L_0^{(K)}}^\times,\Phi_{L/L_0^{(K)}}
^{(\varphi_K^d)}(\varphi_F^{-m}\sigma)\right) \\
& =\left(\pi_K^{m'}.N_{L_0^{(K)}/K}{L_0^{(K)}}^\times,\Phi_{L/L_0^{(K)}}
^{(\varphi_K^d)}(\varphi_K^{-m'}\sigma)\right) \\
&=\pmb{\Phi}_{L/K}^{(\varphi_K)}(\sigma),
\end{aligned}
\end{equation*} 
which completes the proof.
\end{proof}
Finally, the inverse $\pmb{H}_{L/K}^{(\varphi)}
=(\pmb{\Phi}_{L/K}^{(\varphi)})^{-1}$ of the generalized 
Fesenko reciprocity map $\pmb{\Phi}_{L/K}^{(\varphi)}$ defined for 
the extension $L/K$ is the generalization of the Hazewinkel map for 
infinite $APF$-Galois sub-extensions $L/K$ of $K_{\varphi^d}/K$
satisfying $[\kappa_L:\kappa_K]=d$ and under the assumption that the 
local field $K$ satisfies the condition 
given by eq. (\ref{rootofunity}). More precisely, we have the following
proposition.
\begin{proposition}
The following square
\begin{equation*}
\SelectTips{cm}{}\xymatrix{
{K^\times/N_{L_0/K}L_0^\times\times U_{\widetilde{\mathbb X}{(L/K)}}
^\diamond/Y_{L/L_0}}
\ar[r]^-{\pmb{H}_{L/K}^{(\varphi)}}\ar[d]_{(\text{id}_{K^\times/N_{L_0/K}L_0
^\times},\text{Pr}_{\widetilde{K}})} & 
{\text{Gal}(L/K)}\ar[d]^{\txt{mod \text{Gal}(L/K)}'} \\
{K^\times/N_{L_0/K}L_0^\times\times U_{L_0}/N_{L/L_0}U_L}\ar[r]^-{h_{L/K}} & 
{\text{Gal}(L/K)^{ab}}
}
\end{equation*}
is commutative, where
\begin{equation*}
h_{L/K}: K^\times/N_{L_0/K}L_0^\times\times U_{L_0}/N_{L/L_0}U_L\rightarrow
\text{Gal}(L/K)^{ab}
\end{equation*}
is the Hazewinkel map of $L/K$.
\end{proposition}

${}$

\noindent
\textsc{Department of Mathematics, Istanbul Bilgi University, 
Kurtulu\c s Deresi Cad. No. 47, Dolapdere, 34440 Beyo\v glu, Istanbul, Turkey}

\noindent
E-mail : \texttt{ilhan$@$bilgi.edu.tr}

\noindent
\textsc{G\"um\"u\c s Pala Mahallesi, G\"um\"u\c s Sok., \"Oz Aksu Sitesi, 
C-2/39, 34160 Avc\i lar, Istanbul, TURKEY}

\noindent
E-mail : \texttt{erols73$@$yahoo.com}

\end{document}